\documentclass{amsart}
\usepackage[margin=1in]{geometry}
\usepackage{xcolor}
\usepackage[utf8]{inputenc}
\usepackage{amsfonts, amsthm, amssymb,verbatim,amscd,amsmath, blindtext}

\usepackage{multirow}
\usepackage{graphicx}
\usepackage{enumitem}
\usepackage{amssymb}
\usepackage{ytableau}
\usepackage{tikz}
\usepackage{tikz-cd}
\usepackage{float, pifont}
\usepackage{mathtools}
\usepackage[pagebackref,colorlinks=true,citecolor=blue,linkcolor=blue]{hyperref}
\usepackage{parskip}
\usepackage[capitalise, noabbrev]{cleveref}

\usepackage{algorithm}
\usepackage[noend]{algpseudocode}

\newtheorem{theorem}{Theorem}[section]
\newtheorem{lemma}[theorem]{Lemma}
\newtheorem{corollary}[theorem]{Corollary}
\newtheorem{proposition}[theorem]{Proposition}

\newtheorem{notation}[theorem]{Notation}

\theoremstyle{remark}
\newtheorem{remark}[theorem]{Remark}

\theoremstyle{definition}
\newtheorem{definition}[theorem]{Definition}
\newtheorem{example}[theorem]{Example}
\newtheorem{observation}[theorem]{Observation}

\DeclareMathOperator{\lcm}{lcm}

\DeclareMathOperator{\pd}{pdim}
\DeclareMathOperator{\supp}{supp}

\DeclareMathOperator{\sbridge}{sb}

\DeclareMathOperator{\characteristic}{char}

\newcommand{\G}{\mathcal{G}}

\def\G{{\mathcal G}}

\def\G{{\mathcal G}}

\def\1{{\bf 1}}
\def\0{{\bf 0}}

\begin{document}
\title{\textbf{Minimal Cellular Resolutions of Path Ideals}}

\author{Trung Chau}
\address{Department of Mathematics, University of Utah, 155 South 1400 East, Salt Lake City, UT~84112, USA}
\email{trung.chau@utah.edu}

\author{Selvi Kara}
\address{Department of Mathematics, Bryn Mawr College,  906 New Gulph Rd, Bryn Mawr, PA 19010, USA}
\email{skara@brynmawr.edu}

\author{Kyle Wang}
\address{Department of Mathematics \& Statistics, Dalhousie University 
6316 Coburg Rd., PO BOX 15000, Halifax, NS, Canada}
\email{yp259112@dal.ca}

\begin{abstract}
In this paper, we prove that the path ideals of both paths and cycles have minimal cellular resolutions. Specifically, these minimal free resolutions coincide with the Barile-Macchia resolutions for paths, and their generalized counterparts for cycles. Furthermore, we identify edge ideals of cycles as a class of ideals that lack a minimal Barile-Macchia resolution, yet have a minimal generalized Barile-Macchia resolution.
\end{abstract}

\maketitle

\section{Introduction}

It has been a powerful approach to associate a combinatorial object with an algebraic one and study its algebraic properties via combinatorics \cite{eagon1998resolutions,ha2008monomial, katzman2006characteristic, kiani2016castelnuovo, terai2012cohen}. In the spirit of this approach, the algebraic objects of interest in this paper are path ideals, while the combinatorial counterparts are graphs. Specifically, our focus is on studying the path ideals of graphs and their minimal free resolutions by leveraging the underlying structure of the graphs. Central to this work is the use of (generalized) Barile-Macchia resolutions from \cite{chau2022barile}. Such resolutions fall under the umbrella of Morse resolutions, a class of cellular resolutions first introduced by Batzies and Welker in \cite{BW02}, and they generalize  the minimal free resolution constructed in \cite{BM20} for edge ideals of forests. These resolutions are obtained using homogeneous acyclic matchings, a concept from discrete Morse theory. In addition to the Barile-Macchia resolution, other examples of Morse resolutions have been introduced in the literature. One recent example is the \emph{pruned resolutions} from \cite{Pruned20}.

Path ideals, first introduced by Conca and De Negri in \cite{conca1999m}, have been studied for their algebraic properties \cite{Cn, AF18, banerjee2017regularity, BHK10, campos2014depths}. Path ideals can be seen as a generalization of edge ideals, which have been of significant interest in recent years. Let $G=(V,E)$ be a finite, simple graph with the vertex set $V=\{x_1,\ldots, x_n\}$. Associating the vertices of $G$ with the variables in the polynomial ring $R=\Bbbk [x_1,\ldots, x_n]$, where $\Bbbk$ is any field, the edge ideal of $G$ is then generated by monomials of the form $x_ix_j$ for every $\{x_i,x_j\} \in E$. In essence, edge ideals arise from monomials corresponding to edges in $G$, which are inherently paths of length 1. Extending this idea, ideals generated by monomials corresponding to paths of a specified length in $G$ are called the path ideals of $G$.

In this paper, our goal is to construct minimal resolutions for path ideals of paths and cycles. While there is literature  discussing and providing explicit formulas for the (graded) Betti numbers of these ideals \cite{AF18, PathTree2010}, no construction has yet been provided for their minimal resolution. We achieve this by working with (generalized) Barile-Macchia resolutions from \cite{chau2022barile}, thereby expanding the class of ideals for which these resolutions are minimal.

Our two main results are:

\begin{enumerate}
    \item We establish that path ideals of paths have the \emph{bridge-friendly} property (\cref{thm:path_friendly_min}). This property ensures the minimality of Barile-Macchia resolutions as described in \cite{chau2022barile}. Thus, we determine that the path ideals of paths admit a minimal Barile-Macchia resolution.  From this, we derive formulas for their projective dimension and graded Betti numbers, recovering results from \cite{AF18} and \cite{BHK10} (\cref{cor:pdim_path}, \cref{thm:betti_path}).

    \item We shift our focus to path ideals of cycles. In this context, Barile-Macchia resolutions are not always minimal. One instance is the edge ideal of a 9-cycle, \(I_2(C_9)\), which, as indicated in \cite{chau2022barile}, does not have any minimal Barile-Macchia resolution. Nonetheless, we prove that path ideals of cycles have a minimal generalized Barile-Macchia resolution (\cref{thm:main_cycle}). 
\end{enumerate}

Our results on minimal cellular resolutions of path ideals for paths and cycles generalize the findings from \cite{Pruned20}, where it is shown that edge ideals of paths and cycles have minimal pruned resolutions.

Both Barile-Macchia and generalized Barile-Macchia resolutions are induced by  homogeneous acyclic matchings, called Barile-Macchia and generalized Barile-Macchia matchings, respectively. At the heart of the Barile-Macchia matching construction is the comparison of least common multiples of subsets of the minimal generating set $\G(I)$ of a monomial ideal $I$, with respect to a total order on $\G(I)$. An algorithm for producing Barile-Macchia matchings was introduced by the first two authors in \cite{chau2022barile}. In addition,  \texttt{MorseResolutions Macaulay2} package dedicated to Barile-Macchia resolutions were introduced by the first two authors and  O'Keefe in \cite{chau2023morseresolutions}. While the generalized version  adopts the same foundational principle, it extends to multiple total orders on $\G(I)$ as discussed in \cref{sec:cycle}. For further details, refer to \cite{chau2022barile}.

An important concept in relating homogeneous acyclic matchings and free resolutions is \emph{critical cells}: These are subsets of $\G(I)$ that remain untouched by a homogeneous acyclic matching of $I$. In \cite{BW02}, it was shown that these cells are in one-to-one correspondence with the ranks of free modules from (generalized) Barile-Macchia resolutions. Thus, in this paper, we focus on characterizing critical cells of path ideals using \emph{bridges, gaps} and \emph{true gaps} -- simple yet powerful concepts rooted in the graph's structure as introduced in \cite{chau2022barile}. We use these notions to produce minimal free resolutions of path ideals of paths and cycles.

A key observation concerning the critical cells of path ideals of paths is that two distinct critical cells have different least common multiples. This observation proves useful in constructing a critical cell of maximum size, which in turn allows us to deduce a formula for the projective dimension of path ideals of paths. Additionally, this insight is helpful in identifying all critical cells of path ideals of cycles. Such cells consists of critical cells of path ideals of induced paths as well as critical cells whose least common multiple has the largest multidegree  -- for a cycle on $n$ vertices, it would be $x_1\dots x_n$. Consequently, in the paper's final section, our study primarily focuses on identifying the critical cells of the latter type for cycles.

The structure of this paper is as follows. In \cref{sec2}, we revisit essential concepts and results relevant to Morse resolutions, as well as Barile-Macchia matchings and resolutions. In \cref{sec:min_path}, we explore path ideals of paths. We start this section by delving into the characterizations of bridges, gaps, and true gaps specific to paths. With these characterizations in hand, we obtain the bridge-friendliness and affirm the minimality of the associated Barile-Macchia resolution. This paves the way for us to introduce formulas for the projective dimension and to provide a recursive formula for graded Betti numbers. Lastly, in \cref{sec:cycle}, we turn our attention to path ideals of cycles, offering both a review of and insights into the application of generalized Barile-Macchia resolutions. After characterizing the bridges, gaps, and true gaps, we verify the minimality of their generalized Barile-Macchia resolutions by drawing upon our earlier findings on path ideals.

\section{Preliminaries}\label{sec2}

\subsection{Morse resolutions}
Let $I$ be a monomial ideal in the polynomial ring $R = \Bbbk[x_1,\ldots, x_N]$ with a minimal generating set denoted by $\mathcal{G}(I) = \{m_1,\ldots, m_n\}$. We associate to $I$ the Taylor simplex $X$. The vertices of $X$ correspond to the generators of $I$, whereas its cells are labeled by the least common multiple of the labels of their incident vertices. We also associate a directed graph \(G_X=(V,E)\) with this structure, where \(V\) denotes the cells of \(X\), or equivalently, the subsets of \(\mathcal{G}(I)\).
The edge set $E$ consists of directed edges of the form $(\sigma, \sigma ')$ where $\sigma' \subseteq \sigma$ and $|\sigma'| = |\sigma| - 1$. For any subset $A$ of $E$, we define $G_X^A$ as the directed graph having vertex set $V$ and edge set 
\[E(G_X^A) = (E \setminus A) \cup \{(\sigma', \sigma) \mid (\sigma, \sigma') \in A\}.\] 
Essentially, $G_X^A$ is derived from $G_X$ by reversing the direction of edges belonging to $A$.

Central to our discussion on a Morse resolution of $I$ is the notion of homogeneous acyclic matchings from discrete Morse theory, a concept we revisit below.

\begin{definition}\label{def:acyclicmatch}
    A subset \(A \subseteq E\) is called a  \textbf{homogeneous acyclic matching} of \(I\) if the following conditions hold:
    \begin{enumerate}
        \item (\emph{matching}) Each cell appears in, at most, one edge of \(A\).
        \item (\emph{acyclicity}) The graph $G^A_X$ does not contain a directed cycle.
        \item (\emph{homogeneity}) For any edge \((\sigma,  \sigma')\) in \(A\), we have \(\lcm(\sigma) = \lcm(\sigma')\).
    \end{enumerate}
A cell \(\sigma\) that does not appear in any edge of \(A\) is called an \(A\)-\textbf{critical} cell of $I$. In contexts where there is no ambiguity, we simply refer to it as \textit{critical}.
\end{definition}

Recall that $X$ is a $\mathbb{Z}^N$-graded complex that induces a free resolution $\mathcal{F}$ where $\mathcal{F}_r$ is the free $R$-module with a basis indexed by all cells of cardinality \(r\). The differentials \(
\partial_r\colon \mathcal{F}_r \to \mathcal{F}_{r-1},
\) are
defined as
\[
\partial_r(\sigma) = \sum_{\substack{\sigma' \subseteq \sigma,\\|\sigma'| = r-1}} [\sigma:\sigma'] \frac{\lcm(\sigma)}{\lcm(\sigma')} \sigma'.
\]
where \( [\sigma:\sigma'] \) denotes the coefficient of \(\sigma'\) in the boundary of \(\sigma\) and is either \(1\) or \(-1\). The complex  \(\mathcal{F}\) is called the \textit{Taylor resolution} of $R/I$.

Morse resolutions are refinements of the Taylor resolution. Each homogeneous acyclic matching yields a Morse resolution, and these resolutions may coincide. To precisely define the differentials of a Morse resolution, we need to introduce some additional terminology.

Given a directed edge \((\sigma, \sigma') \in E(G^A)\), set
\[
m(\sigma,\sigma')=
\begin{cases}
    -[\sigma':\sigma] & \text{if } (\sigma', \sigma) \in A,\\
     [\sigma:\sigma'] & \text{otherwise}.
\end{cases}
\]
A \textit{gradient path} \(\mathcal{P}\) from \(\sigma_1\) to \(\sigma_t\) is a directed path
\[
\mathcal{P}\colon \sigma_1 \to \sigma_2 \to \cdots \to \sigma_t
\] 
in \(G^A_X\). Similarly, set 
\(
m(\mathcal{P})=m(\sigma_1,\sigma_2) \cdot m(\sigma_2,\sigma_3) \cdots m(\sigma_{t-1}, \sigma_t).
\)

We are now ready to recall Morse resolutions of monomial ideals.

\begin{theorem}\cite[Proposition 2.2, Proposition 3.1,  Lemma 7.7]{BW02} \label{thm:morseres}
    Let \(A\) be a homogeneous acyclic matching of \(I\). Then \(A\) induces a cellular resolution \(\mathcal{F}_A\), where:
    \begin{itemize}
        \item \((\mathcal{F}_A)_r\) is the free \(R\)-module with a basis indexed by all critical cells of cardinality \(r\).
        \item The differentials  \( \partial_r^A:(\mathcal{F}_A)_r\to (\mathcal{F}_A)_{r-1}\) are defined by:
        \[
        \partial_r^A(\sigma) = \sum_{\substack{\sigma' \subseteq \sigma, \\ |\sigma'|=r-1}} [\sigma:\sigma'] \sum_{\substack{\sigma'' \text{ is critical}, \\ |\sigma''|=r-1}} \sum_{\mathcal{P} \text{ is a gradient path } \\ \text{from } \sigma' \text{ to } \sigma''} m(\mathcal{P}) \frac{\lcm(\sigma)}{\lcm(\sigma'')} \sigma''.
        \]
    \end{itemize}
    The resulting (cellular) free resolution \(\mathcal{F}_A\) is called the \emph{Morse resolution} of \(R/I\) associated to \(A\). Furthermore, \(\mathcal{F}_A\) is minimal if for any two $A$-critical cells $\sigma$ and $\sigma'$ with $|\sigma'|=|\sigma|-1$, we have $\lcm(\sigma)\neq \lcm(\sigma')$.
\end{theorem}

\subsection{Barile-Macchia matchings and Barile-Macchia resolutions.}  In this subsection, we recall Barile-Macchia matchings and resolutions. To ensure that this paper is self-contained and accessible, we present relevant definitions and results from \cite{chau2022barile} that are instrumental to our discussions.

Given  \cref{thm:morseres}, the key to producing a Morse resolution of  $R/I$ is finding a homogeneous acyclic matching of $I$. However, systematically crafting such a matching for any monomial ideal is a challenging task.  A recent development in this direction is the Barile-Macchia algorithm,  as presented in \cite[Theorem 2.8]{chau2022barile}, which produces a homogeneous acyclic  matching.  In the context of \cite{chau2022barile}, a matching arising from this algorithm is called a \textit{Barile-Macchia matching}, and the resulting resolution is called a \textit{Barile-Macchia resolution}.

Below, we recall some of the terminology useful for discussing Barile-Macchia resolutions. Throughout this section, we fix a total order \( (\succ_I) \) on \( \G(I) \). For simplicity, we write \( S \setminus s \) (resp, \( S \cup s \)) instead of \( S \setminus \{ s \} \) (resp, \( S \cup \{ s \} \)) where \( S \) is a set and \( s \in S \) (resp, \( s \notin S \)).

Additionally, throughout our discussion, we use the terms ``subsets of \( \G(I) \)" and ``cells of \( I \)" interchangeably. By cells of $I$, we refer to cells of the corresponding Taylor simplex $X$. Lastly, when we express a cell $\sigma$ as  $\sigma=\{m_{i_1}, \ldots, m_{i_k}\}$, we assume that $m_{i_1} \succ_I m_{i_2} \succ_I \cdots \succ_I m_{i_k}$.

\begin{definition}
Let $\sigma$ be a subset of $\G(I)$. A monomial $m$ is called a \textbf{bridge} of $\sigma$ if $m\in \sigma$ and removing $m$ from $\sigma$ does not change the $\lcm$, i.e., $\lcm (\sigma \setminus m) = \lcm (\sigma)$. 
\end{definition}

\begin{notation}
    If $\sigma$ has a bridge, the notation $\sbridge(\sigma)$ denotes the smallest bridge of $\sigma$ with respect to $(\succ_I)$. We set $\sbridge(\sigma)=\emptyset$ if $\sigma$ has no bridges.
\end{notation}

To fully understand the terms and context we discuss, in Algorithm 1 we recall the Barile-Macchia algorithm as outlined in \cite[Algorithm 2.9]{chau2022barile}. The Barile-Macchia Algorithm systematically constructs a matching in $G_X$ based on the concept of bridges within subsets on \(\G(I)\) and a fixed total order \( (\succ_I)\) of $\G(I)$. As it was shown in \cite[Theorem 2.11]{chau2022barile}, this process always produces a homogeneous acyclic matching.

\begin{algorithm}
\caption{Barile-Macchia Algorithm}\label{algorithm1}
\textbf{Input:} A total order $(\succ_I)$ on $\G(I)$. \\
\textbf{Output:} Set of directed edges $A$ in $G_X$.
\begin{algorithmic}[1]
\State $A \leftarrow \emptyset$
\State $\Omega \leftarrow \{\text{all subsets of } \G(I) \text{ with cardinality at least } 3\}$
\While{$\Omega \neq \emptyset$}
    \State Pick $\sigma \in \Omega$ with maximal cardinality
    \State Remove $\{\sigma, \sigma \setminus \sbridge(\sigma)\}$ from $\Omega$
    \If{$\sbridge(\sigma) \neq \emptyset$}
        \State Add edge $(\sigma , \sigma \setminus \sbridge(\sigma))$ to $A$
    \EndIf
\EndWhile
\ForAll{distinct edges $(\sigma, \tau)$ and $(\sigma', \tau')$ in $A$ with $\tau = \tau'$}
    \If{$\sbridge(\sigma') \succ_I \sbridge(\sigma)$}
        \State Remove $(\sigma', \tau')$ from $A$
    \Else
        \State Remove $(\sigma, \tau)$ from $A$
    \EndIf
\EndFor
\Return $A$
\end{algorithmic}
\end{algorithm}

\begin{definition}
Given a Barile-Macchia matching \( A \) of \( I \):
\begin{enumerate}
    \item For any edge $(\sigma,\tau)$ in the final \( A \) from Algorithm \ref{algorithm1} , the cell $\sigma$ is called \textbf{type-2} while the cell $\tau$ is called \textbf{type-1}.
    \item During the execution of Algorithm \ref{algorithm1}, if a directed edge \((\sigma,\tau)\) is added to \( A \), the cell \( \sigma \) is called \textbf{potentially-type-2}. It is important to note that this edge may not persist in the final \( A \) produced by Algorithm \ref{algorithm1}.
\end{enumerate} 
\end{definition}

\begin{remark}\label{rem:sb_pot_type_2}
  If a subset of $\G(I)$ has a bridge, then it is either type-1 or potentially-type-2.
\end{remark}

In earlier discussions, we emphasized the one-to-one correspondence between the ranks of the free \( R \)-modules in a Barile-Macchia resolution of $R/I$ and the \( A \)-critical cells of \( I \). Here, \( A \) represents the Barile-Macchia matching of \( I \) with respect to \( (\succ_I) \). Thus, delving into the nature of \( A \)-critical cells, along with the remaining cells, is crucial for a deeper comprehension of the Barile-Macchia resolution of \( R/I \). It is important to point out that critical cells of \( I \) consists of those cells either left out of \( A \) during the Barile-Macchia Algorithm or added initially but later excluded in the final refinement of the algorithm. We name these critical cells as follows:

\begin{definition}\label{def:criticals}
   Let $\sigma$ be a subset of $\G(I)$. If $\sigma$ is never added to $A$ in any of the steps throughout Algorithm \ref{algorithm1}, it is called  \textbf{absolutely critical}. If $\sigma$ is potentially-type-2 but not type-2 (initially added to $A$ but removed in the final $A$ produced by Algorithm \ref{algorithm1}), we call it {\bf fortunately critical}.
\end{definition}

The following concepts from \cite{chau2022barile} are useful in characterizing critical and non-critical cells of $I$.

\begin{definition}\cite[Definition 2.19]{chau2022barile}
Let $m, m' \in \G(I)$ and $\sigma$ be a subset of $\G(I)$.

\begin{enumerate}
    \item We say that $m$ \textbf{dominates} $m'$ if and only if $m \succ_I m'$.    
    \item The monomial $m$ is called a \textbf{gap} of $\sigma$ if $m \notin \sigma$, and $\lcm(\sigma \cup m) = \lcm(\sigma)$.
    \item The monomial $m$ is called a \textbf{true gap} of $\sigma$ if it is a gap of $\sigma$ and if $\sigma \cup m$ has no new bridges dominated by $m$. The last condition is equivalent to the following: if  $m'$ is a bridge of $\sigma \cup m$ such that $m \succ_I m'$, then $m'$ is  a bridge of $\sigma$. 
\end{enumerate}
\end{definition}

In \cite{chau2022barile}, the first two authors characterized type-1, potentially-type-2, and type-2 cells in terms of bridges and true gaps. We focus primarily on the characterization of potentially-type-2 cells as this is the type of cell that we encounter in our proofs. Recall that, if $\sigma$ is potentially-type-2, then it follows from the definition of potentially-type-2 that $\sigma$ has at least one bridge.

\begin{theorem}\cite[Theorem 2.24 (b)]{chau2022barile}\label{rem:pot-type-2}
A cell $\sigma$ is potentially-type-2 if and only if $\min_{\succ_I } (\sigma)$ is a bridge where
$$\min_{\succ_I }(\sigma)=\min_{\succ_I } \{ \text{bridges of } \sigma, \text{ true gaps of } \sigma\}.$$ In this case, we have $\min_{\succ_I } (\sigma)=\sbridge(\sigma)$.
\end{theorem}

Based on the characterizations of type-1 and potentially-type-2 cells from \cite{chau2022barile}, we have the following characterization of absolutely critical cells which is also the statement of \cite[Corollary 2.28]{chau2022barile}:
 
\begin{corollary}\label{cor:abs_critical}
     A cell is absolutely critical if and only if it has no bridges and no true gaps.
\end{corollary}

In \cite{chau2022barile}, the first two authors introduced a pivotal class of ideals called ``bridge-friendly" for analyzing the minimality of Barile-Macchia resolutions. As established in \cite[Theorem 2.29]{chau2022barile}, bridge-friendliness is a sufficient condition for an ideal to have a minimal Barile-Macchia resolution. We revisit the definition of this class of ideals using the concept of absolutely critical cells.

\begin{definition}\cite[Definition 2.27]{chau2022barile}
   A monomial ideal \(I\) is called \textbf{bridge-friendly} if, for some total order \(\succ_I\) on \(\G(I)\), every potentially-type-2 cell is of type-2 with respect to ($\succ_I$). Equivalently, all \(A\)-critical cells of \(I\) are absolutely critical. Here, \(A\) represents the Barile-Macchia matching of \(I\) with respect to ($\succ_I$).
\end{definition}

\begin{theorem}\cite[Theorem 2.29]{chau2022barile}\label{thm:friendly_min}
  If $I$ is bridge-friendly, then $R/I$ has a minimal Barile-Macchia resolution.   
\end{theorem}

In the next chapter, we study class of ideals that are bridge-friendly.

\section{Minimal free resolutions of path ideals of paths}\label{sec:min_path}

In this section, our primary goal is to investigate the bridge-friendliness and, consequently, the minimal Barile-Macchia resolutions of path ideals of paths.

Fix two integers $p$ and $n$. Consider a path \( L \) on the vertices \(\{x_1,\ldots, x_{n+p-1}\}\). Let \( R = \Bbbk [x_1,\ldots, x_N] \)  with \( N = n + p - 1 \). The \textbf{$p$-path ideal} of \( L \), denoted as \( I_p(L_{N}) \), is generated by monomials in \( R \) corresponding to paths on  \( p \) vertices along \( L \). Explicitly, we have:
\[ I_p(L_{N}) = (x_1x_2\cdots x_p,\ x_2x_3\cdots x_{p+1},\ \cdots, \ x_nx_{n+1}\cdots x_{n+p-1}). \]
\begin{remark}
    Path ideals can be viewed as an extension of edge ideals of graphs. Specifically, the \(2\)-path ideal of a graph coincides with its edge ideal.
\end{remark}

The set of minimal generators of the \(p\)-path ideal of \(L\) is $\{ m_1, m_2, \ldots, m_n \}$ and we denote this set by $\G$. So, 
\[
\G = \{ m_1, m_2, \ldots, m_n \}
\]
where \(m_i := x_i x_{i+1} \cdots x_{i+p-1}\) for each \(1 \leq i \leq n\). Fix a total order \((\succ)\) on \(\G\) such that 
\[
m_1 \succ m_2 \succ \cdots \succ m_{n}.
\]
Throughout the rest of this chapter, our focus is  on the monomial ideal \( I_p(L_{N}) \). In particular, we  examine its Barile-Macchia matching and resolution with respect to the aforementioned total order. For ease of readability, we introduce the following notation, which is consistently employed throughout the paper.

\begin{notation}\label{not:Ms}
  Fix a variable $x_i$ where \(1 \leq i \leq N\). Let \(M_i\) denote the set of all monomials in $\G$ that are divisible by  \(x_i\) where  
  $$M_i=\{m_j :~ \max \{1, i-p+1\} \leq j\leq \min \{i,n\}\}.$$
  This means
  $$M_i = \begin{cases}
      \{m_1,\ldots, m_i\} &\text{ if } i<p, \\
      \{m_{i-p+1}, \ldots, m_{i-1}, m_i\} &\text{ if } p \leq i\leq n,\\
       \{m_{i-p+1}, \ldots,  m_n\} &\text{ if } i>n.
  \end{cases}$$
\end{notation}

\begin{notation}\label{not:distance}
For a monomial $m\in R$, we denote its support by $\supp(m)$. Recall that this is the set of all variables  dividing $m$. For instance,  we have $\supp(m_i)=\{x_i,\ldots, x_{i+p-1}\}$ for  $m_i \in \G$.
\end{notation}

We begin our analysis with the following lemma, which serves as a foundational tool for the classification of bridges, gaps, and true gaps of a cell.

\begin{lemma}\label{lem:tiny}
Let \( \sigma \) be a cell of \( I_p(L_N) \) such that \( m_i \notin \sigma \) for some \( m_i \in \G \). Assume that \( \lcm (\sigma) \) is divisible by \( m_i \). Then there exist monomials \( m_j \in \sigma \cap M_i\)  and \( m_k \in \sigma \cap M_{i+p-1}\) with $1\leq j<i<k\leq n$ such that  \( k - j \leq  p \). 
\end{lemma}

\begin{proof}
   First note that $1<i<n$ since  \( \lcm (\sigma) \) is divisible by \( m_i \) but \( m_i \notin \sigma \). Consider the set of all monomials in $\G$ that are divisible by one of the variables in $\supp (m_i)$. This set can be expressed as the union  $M_i \cup M_{i+p-1}$ where
   $$M_i= \{m_{q},\ldots, m_i\} \text{ and } M_{i+p-1} = \{m_i,\ldots, m_{\ell} \} $$
with $q=\max\{1,i-p+1\}$ and $\ell= \min\{i+p-1,n\}$
as in \cref{not:Ms}. Then, for any  $m_s\in M_i$ and  $m_t\in M_{i+p-1}$, we have $s \leq i\leq t \leq i+p-1$.  Lastly, note that  if $m_s \in M_i$, then  $m_{s+p} \notin M_i$ which implies that $i< s+p$.

   Since  $m_i$ divides $\lcm (\sigma)$ but $m_i \notin \sigma$,  there exist monomials $m_s \in \sigma \cap M_i $ and $m_t \in \sigma \cap M_{i+p-1} $ with $s<i<t$. Let $j$ be the largest index of a monomial in $\sigma \cap M_i$  and $k$ be the smallest index of monomial in $\sigma \cap M_{i+p-1} $. In other words, pick the monomials $m_j$ and $m_k$ in $\sigma$ that are closer to $m_i$ from either direction. Observe that   $m_i$ divides $\lcm(m_j, m_k)$. Otherwise, there exists $m_s\in \sigma \cap M_i$ with $j<s$ or $m_t\in \sigma\cap M_{i+p-1}$ with $t<k$  which contradicts to either maximality of $j$ or minimality of $k$.   
   
   Lastly, we show $k-j\leq p$. If $k -j > p$, then $x_{j+p}$ does not divide $\lcm(m_j, m_k)$.  However, $x_{j+p}$  divides $m_i$ since $i+1\leq  j+p<k\leq i+p-1$ from the first paragraph. This leads to a contradiction as $m_i$ divides $\lcm(m_j, m_k)$. Thus,  we must have $k - j \leq p$.
\end{proof}

In the following result, we provide a characterization of bridges and gaps of a cell. It is important to note that neither \( m_1 \) nor \( m_{n} \) can be a gap or a bridge of any cell.

\begin{proposition}\label{prop:bridge_gap}
Let \( \sigma \) be a cell of \( I_p(L_N) \). For a monomial \( m_i \in \G \) with \( 1 < i < n \), the following statements hold:
    \begin{enumerate}
        \item monomial \( m_i \) is a bridge of \( \sigma \) if and only if \( m_i \in \sigma \) and there exist monomials \( m_j \in \sigma \cap M_i \) and \( m_k \in \sigma \cap M_{i+p-1} \) such that $j<i<k\leq n$ and \( k - j \leq  p \).
        \item monomial \( m_i \) is a gap of \( \sigma \) if and only if \( m_i \notin \sigma \) and there exist monomials \( m_j \in \sigma \cap M_i \) and \( m_k \in \sigma \cap M_{i+p-1} \) such that  $j<i<k\leq n$ and \( k - j \leq  p \). 
    \end{enumerate}
\end{proposition}

\begin{proof}
    If $m_i$ is a bridge of $\sigma$, then application of \cref{lem:tiny} to $\sigma \setminus m_i$ results with the desired conditions. For the other direction, if the conditions are met, then  $m_i$ divides $\lcm(m_j,m_k)$. Since $\lcm(m_j,m_k)$ divides $\lcm (\sigma)$, monomial $m_i$ must be a bridge of $\sigma$. Characterization of gaps in part (2) follows similarly.
\end{proof}

\begin{remark}\label{rem:closest_j_k}
  When $m_i$ is a bridge or a gap of a cell $\sigma$, there can be several $(m_j,m_k)$ pairs of $m_i$ where $m_j$ and $m_k$ are as in the statement of \cref{prop:bridge_gap}. However, there is only one $(m_j,m_k)$ pair if  $m_j$ and $m_k$ are chosen to be the closest to $m_i$ as in the proof of \cref{lem:tiny}. 
\end{remark}

Next, we present a characterization of true gaps in terms of  these $(m_j,m_k)$ pairs.  In particular,  we show that $m_k=m_{k'}$  for any two pairs $(m_j,m_k)$ and $(m_{j'}, m_{k'})$ of a true gap as in \cref{prop:bridge_gap}. So,  there is a unique such $m_k$ for true gaps.

\begin{proposition}\label{prop:path_truegap}
     Let \( \sigma \) be  a cell of \( I_p(L_N) \). Consider a monomial \( m_i \in \G \) for $1<i<n$ such that  $m_i$ does not dominate any bridges of $\sigma$. Then  $m_i$ is a true gap of $\sigma$ if and only if the following statements hold:
        \begin{enumerate}[label=(\arabic*)]
            \item[(a)] Monomial $m_i$ is a gap of $\sigma$. 
            \item[(b)] There is only one monomial in $\sigma \cap M_{i+p-1}$. 
            \item[(c)]  If $i+p\leq n$, then $m_{i+p}\notin \sigma$.
        \end{enumerate}      
\end{proposition}

\begin{proof}
Recall the assumption that $m_i$ does not dominate any bridges of $\sigma$. In addition, recall the following result from \cite[Proposition 2.21]{chau2022barile} which will be useful in the proof:  \( m_i \) is a true gap of $\sigma$ that does not dominate any bridges of $\sigma$ if and only if  \( m_i \)  is a gap of \( \sigma \) and \( \sbridge(\sigma \cup m_i) = m_i \).

For the forward direction, suppose \( m_i \) is a true gap of \( \sigma \). It then follows from \cite[Proposition 2.21]{chau2022barile} that \( \sbridge(\sigma \cup m_i) = m_i \).

\begin{enumerate}[label=(\alph*)]
    \item By the definition of a true gap, \( m_i \) is a gap of $\sigma$. 

    \item It follows from part (a) and  \cref{prop:bridge_gap} (2) that there exists a monomial $m_k \in \sigma \cap M_{i+p-1}$  for $i<k$.     If there exists another monomial $m_t \in \sigma \cap M_{i+p-1}$, then  \( i<k , t \leq i+p-1\) where the last inequality is due to \cref{not:Ms}. We may assume $k<t$ (otherwise, switch $k$ and $t$ in the following arguments). This means $\supp(m_k) \subseteq \supp(m_i)\cup \supp(m_t)$ which is equivalent to  \( m_k \) divides $\lcm(m_i,m_t)$. Since  $\lcm(m_i,m_t)$ divides $\lcm(\sigma \cup m_i)$, monomial $m_k$ is a bridge of \( \sigma \cup m_i \) where  $m_i \succ m_k$. This contradicts \( \sbridge(\sigma \cup m_i) = m_i \). So, we have \( \sigma \cap M_{i+p-1}=\{m_k\} \).

    \item Suppose $i+p\leq n$. Let $m_k$ be the monomial from part (b). If $m_{i+p}\in \sigma$, then  $\supp(m_k) \subseteq \supp(m_i)\cup \supp(m_{i+p})$ since $i<k<i+p$. This implies that  $m_k$  is a bridge of $\sigma \cup m_i$ where $m_i \succ m_k$, a contradiction as \( \sbridge(\sigma \cup m_i) = m_i \). Hence,   $m_{i+p}\notin \sigma$.
\end{enumerate}

For the other direction, assume  (a), (b), and (c) hold. For contradiction, suppose that $m_i$ is not a true gap of $\sigma$. Let $m_t:= \sbridge(\sigma\cup m_i)$. Then, by \cite[Proposition 2.21]{chau2022barile}, we have $m_i\succ   m_t$ which implies that  $i<t$ and $m_t\in \sigma$. Although $m_t$ is a bridge of $\sigma\cup m_i$, notice that $m_t$ cannot be a bridge of $\sigma$ since $m_i$ does not dominate any bridges of $\sigma$ by the initial assumption.

Since $m_t$ is a bridge of $\sigma\cup m_i$, by \cref{prop:bridge_gap} (1),  there are monomials $m_{j_t}, m_{k_t} \in \sigma \cup m_i$ such that $j_t < t < k_t\leq n$ and $k_t-j_t \leq p$. Since $m_t$ is not a bridge of $\sigma$, either $j_t=i$ or $k_t=i$. Given \( i < t \), it follows that \( j_t = i \). Consequently, we have \( i < t < k_t \leq i+p \) which implies that \( m_t \in  \sigma  \cap M_{i+p-1} \). Since $ m_{k_t} \in \sigma$, it follows from  (b)  that  \( m_{k_t} \notin \sigma \cap M_{i+p-1} \), meaning that $ k_t\geq i+p$. This means \( k_t = i+p \leq n\), which in turn implies \( m_{k_t} = m_{i+p} \in \sigma\), contradicting (c). Therefore, \( m_i \) must indeed be a true gap of \( \sigma \). \qedhere
\end{proof}

\begin{example}\label{examplepathideal}
    Consider the $3$-path ideal of an $8$-path,
    \[I=(x_1x_2x_3,x_2x_3x_4,x_3x_4x_5,x_4x_5x_6,x_5x_6x_7,x_6x_7x_8).\]
    Consider the subset $\sigma=\{m_1,m_4,m_6\}$. It  has no bridges, and its gaps are $m_2, m_3$, and $m_5$. Using \cref{prop:path_truegap}, we identify which among these gaps are true gaps.

\begin{itemize}
    \item The monomial $m_2$ is a true gap of $\sigma$. This is confirmed by: (a) the observation that $m_2\notin \sigma$ and $m_1,m_4\in \sigma$, with the difference in their indices satisfying $4-1\leq 3$; (b) within the set $M_{4}=\{m_2,m_3,m_4\}$, only $m_4$ belongs to $\sigma$; and (c) the monomial $m_{2+3}=m_{5}$ is not in $\sigma$.
    \item The monomial $m_3$ is not a true gap of $\sigma$. This is due to the failure of part (c) of \cref{prop:path_truegap} (3), given that $m_6\in \sigma$.
    \item The monomial $m_5$ is  a true gap of $\sigma$ based on: (a) $m_5\notin \sigma$ and both $m_4$ and $m_6$ are in $\sigma$, satisfying $6-4\leq 3$; (b) from the set $M_{7}=\{m_5,m_6\}$, only $m_6$ is in $\sigma$. Furthermore, (c) is not applicable since $8>6$, and thus, $m_{5+3}=m_{8}$ does not exist.
\end{itemize}
\end{example}

In what follows, we show that  path ideals of paths are bridge-friendly. 

\begin{theorem}\label{thm:path_friendly_min}
   The path ideal $I_p(L_{N})$ is bridge-friendly, and its Barile-Macchia resolution is ~minimal.
\end{theorem}

\begin{proof}
   It suffices to show that $I_p(L_{N})$ is bridge-friendly since the minimality of the Barile-Macchia resolution is a direct consequence of \cref{thm:friendly_min}.
   
First recall \cite[Lemma 2.33]{chau2022barile} which will be useful in deducing bridge-friendliness:   $I_p(L_{N})$ is bridge-friendly if and only if, for any potentially-type-2 cell $\sigma$ (should it exist), there is no monomial $m \in \G$ such that $m$ is a true gap of $\sigma \setminus \sbridge(\sigma)$ and  $\sbridge(\sigma)\succ m$. 
    
    If $I_p(L_{N})$ has no potentially-type-2 cells, then its Taylor resolution is inherently minimal, making the path ideal bridge-friendly. So, we may assume $I_p(L_{N})$ has potentially-type-2 cells.  Let $\sigma$ be a potentially-type-2 cell of $I_p(L_{N})$. By definition, \( \sbridge(\sigma) \) exists. Since $ \sbridge(\sigma) \neq m_n$ as  $m_n$ cannot be a bridge of $\sigma$, there exists a monomial  $m_i \in \G$ satisfying $\sbridge(\sigma) \succ m_i$.  It is possible to have $m_i=m_n$.

    To utilize \cite[Lemma 2.33]{chau2022barile}, our goal is to show  that $m_i$ cannot be a true gap of $\sigma \setminus \sbridge(\sigma)$. We prove this by contradiction. If $m_i$ is a true gap of $\sigma \setminus \sbridge(\sigma)$, we verify below that $m_i$ is a true gap of $\sigma$. Since $\sbridge(\sigma) \succ m_i$, $\sigma$ has a bridge dominating a true gap,  contradicting \cref{rem:pot-type-2} (b)   as $\sigma$ is potentially-type-2. Thus, no such $m_i$ exists. Therefore, the path ideal  $I_p(L_{N})$ is bridge-friendly.
     
   Now we prove $m_i$ is a true gap of $\sigma$ when  $m_i$ is a true gap of $\sigma \setminus \sbridge(\sigma)$. Since $\sbridge(\sigma) \succ m_i$, monomial $m_i$ does not dominate any bridges of $\sigma$. Thus, we can apply \cref{prop:path_truegap} to $\sigma \setminus \sbridge (\sigma)$ and conclude that
   \begin{itemize}
       \item[(a)] $m_i$ is a gap of $\sigma$; 
       \item[(b)] there is only one monomial $m_k$ in $\sigma  \cap M_{i+p-1}$ since $\sbridge(\sigma) \succ m_i \succ m_k$; 
       \item[(c)]   if $i+p\leq n$, then $m_{i+p}\notin \sigma \setminus \sbridge (\sigma)$ $\iff $ $m_{i+p}\notin \sigma $ since $m_{i+p}\neq  \sbridge(\sigma) $. This is because   $m_i \succ m_{i+p}=\sbridge(\sigma)$  contradicts to $\sbridge(\sigma) \succ m_i$.
   \end{itemize}
 Hence, $m_i$ is a true gap of $\sigma$ by \cref{prop:path_truegap}.
   \end{proof}

In the subsequent discussion, our goal is to demonstrate that for every multidegree $m$, i.e., a monomial in \( R \), there exists at most one critical cell \( \sigma \) such that \( \lcm(\sigma) = m \). Establishing the uniqueness of this critical cell for each multidegree allows us to compute both the projective dimension and Betti numbers of the path ideal via its Barile-Macchia resolution. To pave the way for this claim, we first introduce several auxiliary lemmas.

Recall from \Cref{cor:abs_critical} that the critical cells of \( I_p(L_N ) \) have no bridges and no true gaps, a consequence of  being bridge-friendly.

\begin{lemma}\label{lem:aux2}
Let $\sigma$ be a critical cell of $I_p(L_N)$. Let   $a, b \in \mathbb{N}$ such that  $p<a \leq b-p-1$  and 
$$M=M_a\cup \cdots M_{b-p-1}=\{m_{a-p+1},m_{a-p+2}, \ldots, m_{b-p-1}\}.$$  Assume that  $\sigma \cap M = \{m_{a-p+1}\}$. Then the following statements hold:
\begin{enumerate}
    \item  $\sigma \cap \{m_{a-2p+1}, \ldots, m_{a-p-1}\}= \emptyset $.
    \item $ \lcm(\sigma)$ is divisible by $x_{a-p}$ if and only if $\sigma$ contains $m_{a-p}$.
\end{enumerate}
\end{lemma}

\begin{proof} 
(1) Suppose $\sigma \cap \{m_{a-2p+1}, \ldots, m_{a-p-1}\} \neq  \emptyset $.  Under this assumption, we show that \( m_{a-p} \) is either a bridge or a true gap of \( \sigma \) which leads to a contradiction by \Cref{cor:abs_critical}  since $\sigma$ is a critical cell of a bridge-friendly ideal.

Since we assumed (1) fails, there exists a monomial \( m_j\) in \( \sigma \cap  \{m_{a-2p+1}, \ldots, m_{a-p-1}\}\). Notice that $m_j \in \sigma \cap M_{a-p}$ and $m_{a-p+1}\in \sigma \cap M_{a-1}$ where $j< a-p < a-p+1$ with $(a-p+1)-j \leq p$. By setting $i=a-p$ and $k=a-p+1$, it follows  from \Cref{prop:bridge_gap} that the monomial $m_{a-p}$ is either a bridge or a gap of $\sigma$.

If $m_{a-p}$ is a bridge of $\sigma$, we are done. Now, assume $m_{a-p}$ is a gap of $\sigma$. Since $\sigma$ is a critical cell, it has no bridges which means $m_{a-p}$ does not dominate any bridges of $\sigma$. This allows us to use \cref{prop:path_truegap} to conclude that $m_{a-p}$ is a true gap of $\sigma$, completing the proof of (1). The following is a verification of \cref{prop:path_truegap} (a)-(c):
        \begin{enumerate}
            \item[(a)] \( m_{a-p} \) is a gap of \( \sigma \) by our assumption.
            \item[(b)] Recall that $M_{a-1}=  \{m_{a-p}, m_{a-p+1}, \ldots, m_{a-1}\}$ and $m_{a-p}\notin \sigma$. Then $M_{a-1}\setminus \{m_{a-p}\} \subseteq M $ and $\sigma \cap M=\{m_{a-p+1}\}$ implies that $\sigma \cap   M_{a-1}=\{ m_{a-p+1} \}$.
            \item[(c)] Since $a\leq n$, we need to show $m_a\notin \sigma$. This follows from  $m_a \in M$ and $\sigma \cap M=\{m_{a-p+1}\}$.
        \end{enumerate}

 (2) It is immediate that  \( m_{a-p} \in \sigma \) implies  \( x_{a-p} \) divides \( \lcm(\sigma) \). On the other hand, if \( \lcm(\sigma) \) is divisible by \( x_{a-p} \), then  $\sigma\cap M_{a-p} \neq \emptyset$. This means   $\sigma\cap M_{a-p}= \{m_{a-p}\}$ by part (1). Thus, we have $m_{a-p}\in \sigma$.
\end{proof}

\begin{lemma}\label{lem:imp}
 For a critical cell $\sigma$  of $I_p(L_N)$,   let  $b \in \mathbb{N}\cup\{\infty\}$ such that  $\{i :~ x_i\mid \lcm(\sigma) \text{ and } i\leq b-p-1 \}\neq \emptyset$. Define   \(b_1\) as:
    \[
    b_1 \coloneqq 
    \max \{i :~ x_i\mid \lcm(\sigma) \text{ and } i\leq b-p-1 \}. 
    \] 
   Assume  $\sigma \cap \{m_{b-2p+1}, \ldots, m_{b-p-1}\} =\emptyset$ when $b\in \mathbb{N}$. Then $ \sigma \cap M=\{m_{b_1-p+1}\}$ and $b_1\geq p$ where 
    $$M= M_{b_1}\cup M_{b_1+1} \cup \cdots \cup M_{b-p-1}= \{m_{b_1-p+1}, m_{b_1-p+2},\ldots, m_{b-p-1}\}. $$
\end{lemma}

\begin{proof}
 We analyze the two possible cases for \( b \) separately. 

\textbf{Case 1:} \( b = \infty \). In this case, we have $M= M_{b_1} \cup M_{b_1+1} \cup \cdots M_n$. By the definition of \( b_1 \), it is clear that \( b_1 \) is finite. Given this,  \( \lcm(\sigma) \) is not divisible by \( x_k \) for any integer \( k \geq b_1+1 \). Thus, \( \sigma \cap (M \setminus \{m_{b_1-p+1}\})= \emptyset \).  Moreover, as \( x_{b_1} \) divides \( \lcm(\sigma) \), we have $\sigma \cap M_{b_1}\neq \emptyset$  which guarantees that \(\sigma \cap M=\{ m_{b_1-p+1} \} \) and $b_1\geq p$.

\textbf{Case 2:} \( b \in \mathbb{N} \). 
It follows from the definition  that \( b_1 \leq b-p-1 \). If \( b_1 < b-p-1 \), then \( \lcm(\sigma) \) is divisible by $x_{b_1}$ but it is not divisible by any of the variables  among \( \{x_{b_1+1},\ldots, x_{b-p-1}\} \). This means     $\sigma \cap M_{b_1} \neq \emptyset$ but \( \sigma \cap (M \setminus M_{b_1} )=\emptyset\). Thus, \( \sigma \cap M= \{ m_{b_1-p+1}\} \) and $b_1\geq p$.

If \( b_1 = b-p-1 \),  then the assumption $\sigma \cap \{m_{b-2p+1},\ldots, m_{b-p-1}\} =\emptyset$ becomes $\sigma \cap (M_{b_1} \setminus \{m_{b_1-p+1}\}) = \emptyset$ where $M=M_{b_1}$. Since \( x_{b_1} \) divides \( \lcm(\sigma) \), we have  $\sigma \cap M_{b_1} \neq \emptyset $.  Hence,  $m_{b_1-p+1}$ is the only monomial in  $\sigma \cap M$.
\end{proof}

Building upon the preceding  lemmas, we introduce a sequence of results that are key to understanding the monomials $m \in \G$ that are in  a given critical subset \( \sigma \). We first introduce a new terminology and a few immediate observations that will be used in the next few results.

\begin{definition}
    For a critical cell $\sigma$ of $I_p(L_N)$, define a sequence $(b_0,b_1,b_2, \ldots)$ by setting $b_0=\infty$ and, for $j \geq 1$, 
    \[b_j \coloneqq \max \{i :~ x_i \mid \lcm(\sigma) \text{ and } i \leq b_{j-1} - p - 1\}.\]
    If no such $x_i$ exists, set $b_j = -\infty$.  This sequence  $(b_0,b_1,b_2, \ldots)$ is called the {\bf $\sigma$-sequence}. Notice that for a critical cell $\sigma$, its $\sigma$-sequence is uniquely defined.
\end{definition}

\begin{example}
    Consider the $3$-path ideal of an $8$-path where $n=6$ and $p=3$:
    \[I=(x_1x_2x_3,x_2x_3x_4,x_3x_4x_5,x_4x_5x_6,x_5x_6x_7,x_6x_7x_8).\]
   and the cell $\sigma=\{m_1,m_2,m_5,m_6\}$. One can use   \cref{prop:bridge_gap} and \cref{prop:path_truegap} to show $\sigma$ is a critical cell of $I$. Computation of the $\sigma$-sequence $(b_1,b_2,\ldots)$ relies on  $\lcm(\sigma)=x_1x_2\cdots x_8$.  
By the set-up, we have $b_0=\infty$ and 
        \begin{align*}
           b_1 &= \max \{i :~ x_i \mid \lcm(\sigma)\}=8 \\
           b_2 &=\max \{i :~ x_i \mid \lcm(\sigma) \text{ and } i \leq b_1-4=4 \}=4
        \end{align*}
  Since $\{i :~ x_i \mid \lcm(\sigma) \text{ and } i \leq b_2-4=0 \}=\emptyset$, we have $b_3=-\infty$.
\end{example}

As the above example indicates, each $\sigma$-sequence is finite. We discuss this more in detail in the following observation.
\begin{observation}
    Let  $\sigma$ be a critical cell of $I_p(L_N)$. 
    \begin{enumerate}
        \item[(a)]  Observe that the  $\sigma$-sequence $(b_0,b_1,b_2, \ldots)$ is finite. To see this, first note that $b_1$  is the largest index of a variable in the support of $\lcm(\sigma)$. Next, notice that   the sequence $(b_0,b_1,b_2,\ldots)$ strictly decreases  after $b_1$  since $b_j \leq b_{j-1}-p-1$ for $j\geq 2$. Since $\supp(\lcm(\sigma))$ is finite, the $\sigma$-sequence  eventually reaches $- \infty$ after finitely many steps, i.e., there exists an $\ell \geq 1$ for which 
$$\{i :~ x_i\mid \lcm(\sigma) \text{ and } i\leq b_{\ell-1}-p-1 \}\neq \emptyset,$$
$$ \{i :~ x_i\mid \lcm(\sigma) \text{ and } i\leq b_{\ell}-p-1 \}=\emptyset.$$ This means $b_{\ell}$ is finite and $b_{\ell+1} = -\infty$. 
    \end{enumerate}
  From now on, we write the $\sigma$-sequence as a finite sequence $(b_0,b_1,\ldots, b_{\ell})$ where each $b_i \in \mathbb{N}$ is non-zero.
  \begin{enumerate}
      \item[(b)]  Since  $b_j\leq b_{j-1}-p-1$ for each $1\leq j\leq \ell$, we have
      $$  (\ell-j) p + (\ell-j+1) \leq b_j<  \cdots <b_1 \leq n+p-1.$$
  \end{enumerate}
\end{observation}

Now, we can determine which monomials in $\G$ are part of a critical cell $\sigma$, using the $\sigma$-sequence as our key tool.

\begin{proposition}\label{prop:critical}
Let  $\sigma$ be a critical cell of $I_p(L_N)$ with its $\sigma$-sequence $(b_0,b_1, \ldots, b_{\ell})$. Let
\begin{align*}
   M_{b_j,b_{j-1}} &\coloneqq M_{b_j}\cup M_{b_j+1}\cup \cdots \cup M_{b_{j-1} -p-1}\\
   & = \{m_{b_j-p+1},m_{b_j-p+2}, \ldots, m_{b_{j-1}-p-1}\}.
\end{align*}
Then, we have $\sigma \cap M_{b_j,b_{j-1}} =\{m_{b_j-p+1}\}$ when  $1 \leq j \leq \ell$. Moreover, $m_k\notin \sigma$ for  $k \leq b_{\ell} - p - 1$.

\end{proposition}

\begin{proof}
We use induction on $j$ where $1 \leq j \leq \ell$.  The base case $j=1$ is covered in \cref{lem:imp}.  

For a fixed $j < \ell$,  suppose the statement holds for each $ k \in \{1,\dots, j\}$, that is,  $\sigma \cap M_{b_k,b_{k-1}}=\{ m_{b_k-p+1}\}$.   Since   $\sigma \cap M_{b_j,b_{j-1}}=\{m_{b_j-p+1}\}$  by the induction hypothesis and $ p<   b_j \leq b_{j-1}-p-1$,  we can apply  \cref{lem:aux2} to obtain   \begin{equation*}
  \sigma \cap \{m_{b_j-2p+1}, \ldots, m_{b_{j}-p-1}\}=\emptyset.  
\end{equation*} 
The quality above and the fact that  $\{i :~ x_i\mid \lcm(\sigma) \text{ and } i\leq b_j-p-1 \}\neq \emptyset$ allow us to utilize  \cref{lem:imp}. Then, we conclude that $\sigma \cap M_{b_{j+1},b_{j}}=\{m_{b_{j+1}-p+1}\}$, as desired.

For the final part,  note that  $\lcm(\sigma)$ is not divisible by any variable $x_k$ with $k \leq b_{\ell} - p - 1$. This implies that $\sigma \cap M_k=\emptyset$  for any such $k$. Hence, $m_k\notin \sigma$ for  $k \leq b_{\ell} - p - 1$.
\end{proof}

\begin{proposition}\label{prop:content}
    If $\lcm(\sigma) = \lcm(\sigma')$ for two critical cells $\sigma$ and $\sigma'$ of $I_p(L_N)$, then $\sigma = \sigma'$. 
\end{proposition}

\begin{proof}
  Let $\sigma$ and $\sigma'$ be two critical cells of $I_p(L_N)$ where $(b_0,b_1,\ldots, b_{\ell})$ is the $\sigma$-sequence and  $(b'_0,b'_1,\ldots, b'_k)$ be the $\sigma'$-sequence. Recall that the value of each $b_i$ and $b'_j$ are determined solely by the least common multiples.  Since $\lcm(\sigma) = \lcm(\sigma')$, we must have $(b_0,b_1,\ldots, b_{\ell})= (b'_0,b'_1,\ldots, b'_k)$. 

  Next, notice that $\G$ can be written as the disjoint union of the following three sets:
  $$\G = \{m_1, \ldots, m_{b_{\ell} -p}\} \sqcup \Bigg( \bigcup_{j=1}^{\ell}  M_{b_{j}, b_{j-1}} \Bigg) \sqcup \Bigg( \bigcup_{j=1}^{\ell} \{ m_{b_{j}-p}\} \Bigg).$$
 It follows from \cref{prop:critical} that none of the monomials in the first set are contained in \( \sigma \) or \( \sigma' \). Moreover, the only monomials in the second set that appear in both \( \sigma \) and \( \sigma' \) are those of the form \( m_{b_j - p +1} \) for each \( 1 \leq j \leq \ell \). For the last set, we apply \cref{lem:aux2} (2) and conclude that \( m_{b_j - p} \) is contained in the critical cell \( \tau \) if and only if \( x_{b_j - p} \) divides \( \lcm(\tau) \), where \( \tau \in \{\sigma, \sigma'\} \) for \( 1 \leq j \leq \ell \). Therefore,   $\sigma = \sigma'$. 
\end{proof}

From the preceding proposition, we deduce that the least common multiples of distinct critical cells of \(I_p(L_N)\) are different. This particularly implies the following information on its multi-graded Betti numbers.

\begin{corollary}\label{thm:path_min}
    For any monomial $m$ and any integer $i$, we have
    \[
    \beta_{i,m}(R/I_p(L_N))=\begin{cases}
        1 & \text{ if there exists a critical subset of cardinality } i \text{ whose lcm equals } m,\\
        0 & \text{ otherwise.}
    \end{cases}
    \]
\end{corollary}

From the characterization of critical cells in  \cref{prop:critical} and the insights from the proof of  \cref{prop:content}, we can deduce the projective dimension of \(I_p(L_{n+p-1})\).

\begin{corollary}\label{cor:pdim_path}
      Let \( n \) be expressed as \( n = (p+1)q + s \) where \( 0 \leq s \leq p \). The projective dimension of \( R/ I_p(L_{n+p-1}) \) is given by:
    \[
    \pd(R/I_p(L_{n+p-1})) =
    \begin{cases}
    2q & \text{if } s = 0,\\
    2q + 1 & \text{if } s = 1,\\
    2q + 2 & \text{otherwise}.
    \end{cases}
    \]  
\end{corollary}

\begin{proof}
Among every collection of \(p+1\) consecutive monomials, at most two can be in a critical cell by \cref{prop:critical}. Namely,  among the following monomials $m_{b_j-p} , m_{b_j-p+1}, \ldots, m_{b_j}$ for $1\leq j \leq \ell$, we can have at most $m_{b_j-p} , m_{b_j-p+1} \in \sigma$. This allows us to construct cells of maximal cardinality as follows:
\begin{itemize}
    \item when $s=0$, set $\sigma_1=\{m_p,m_{p+1},\ldots, m_{n-p-2},m_{n-p-1},  m_{n-1},m_n\}$ where $|\sigma_1|=2q$,
    \item when $s=1$, set $\sigma_2=\{m_1,m_{p+1},m_{p+2},\ldots, m_{n-p-2},m_{n-p-1},  m_{n-1},m_n\}$ where $|\sigma_2|=2q+1$,
    \item when $2\leq s\leq p$, set $\sigma_3=\{m_{s-1},m_s, m_{p+s},m_{p+s+1},\ldots, m_{n-p-2},m_{n-p-1},  m_{n-1},m_n\}$ where $|\sigma_3|=2q+2$.
\end{itemize}
One can verify that each of these cells $\sigma_1, \sigma_2$ and $\sigma_3$ are absolutely critical as  neither of them have bridges or true gaps by \cref{prop:bridge_gap} and \cref{prop:path_truegap}. 
Consequently, we can derive the maximal cardinality of a critical cell for each case, thus obtaining the projective dimension. 
\end{proof}
A formula for the projective dimension of the path ideal of a path was also given in \cite{AF18}. The formula from \cite{AF18} matches ours. However, while we express the projective dimension based on the number of minimal generators of the $p$-path ideal, \cite{AF18} does so using the length of the path.

We also recover the recursive formula for graded Betti numbers from \cite{BHK10}. This formula was utilized in \cite{AF18} to provide explicit calculations for the projective dimension and regularity of path ideals of paths and cycles.

\begin{theorem}\label{thm:betti_path}
    For all indices \( r,d \), we have
    \[ \beta_{r,d}(R/I_p(L_n)) = \beta_{r,d}(R/I_p(L_{n-1})) +\beta_{r-1,d-{p}}(R/I_p(L_{n-(p+1)}))+\beta_{r-2,d-(p+1)}(R/I_p(L_{n-(p+1)})). \]
\end{theorem}
\begin{proof}
    By  \cref{thm:path_min}, \( \beta_{r,d}(R/I_p(L_n)) \) counts the critical cells of cardinality \( r \) and degree \( d \). To derive our desired expression, it suffices  to partition the set of critical subsets of cardinality $r$ and degree $d$ in an appropriate way.  Recall that $\G(I_p(L_n)) = \{m_1,\ldots, m_{n-p+1}\}$ and neither $m_1$ nor $m_{n-p+1}$ can be a bridge or a true gap of any cell of $I_p(L_n)$. Consider a cell \( \sigma \) of $I_p(L_n)$.  The following three scenarios for \( \sigma \) completes the proof:

\textbf{Case 1:} Suppose \( m_{n-p+1} \notin \sigma \). In this case, \( \sigma \) is a cell of both \( I_p(L_n) \) and \( I_p(L_{n-1}) \). Our goal is to show the following:  
\( \sigma \) is a critical cell of \( I_p(L_n) \) if and only if it is a critical cell of \( I_p(L_{n-1}) \).  

Notice that any critical cell of \( I_p(L_n) \) is automatically a critical cell of \( I_p(L_{n-1}) \). Conversely, any critical cell of \( I_p(L_{n-1}) \) must also be a critical cell of \( I_p(L_n) \), since \( m_{n-p+1} \) cannot be a true gap—indeed, it cannot be a gap at all. This completes the proof.

\textbf{Case 2:} Suppose \( m_{n-p+1} \in \sigma \) but \( m_{n-p} \notin \sigma \). Our goal is to show the following:  $\sigma$ is a critical cell of  $I_{p}(L_{n})$ if and only if $\sigma \backslash \{m_{n-p+1}\}$ is a critical cell of $I_{p}(L_{n-(p+1)})$. Recall that $\G(I_p(L_{n-(p+1)})) = \{m_1,\ldots, m_{n-2p}\}$ and neither $m_1$ nor $m_{n-2p}$ can be a bridge or true gap of any cell of $I_{p}(L_{n-(p+1)})$.

Suppose $\sigma$ is a critical cell of  $I_{p}(L_{n})$. Let $\tau=\sigma \backslash \{m_{n-p+1}\}$. We first show that $\tau$ is a cell of $I_{p}(L_{n-(p+1)})$, i.e.  there exists  no $m_j \in \sigma$ for  $n-2p< j < n-p$. If there is at least one such $m_j \in \sigma$, then $m_{n-p}$ is a true gap of $\sigma$ by \cref{prop:path_truegap} as we explain in the following steps: (a)  $m_j, m_{n-p+1} \in \sigma$ and $(n-p+1)-j \leq p$; (b)  $\sigma \cap M_{n-1}=\{m_{n-p+1}\}$;  (c)  $n>n-p+1$. Since $\sigma$ is critical, it cannot have a true gap.  Thus, $m_j \notin \sigma$.

Next, we show that \( \tau \) is critical. Note that \( \tau \) has no bridges; otherwise, \( \sigma \) would have a bridge, which is impossible since \( \sigma \) is critical. If $\tau$ has a true gap, say $m_i$,  then  $m_i$ cannot dominate any bridges of $\tau$ since it has none. Then, we apply \cref{prop:path_truegap}   to $m_i$ and $\tau$ and conclude that $m_i$ is a true gap of $\sigma$, a contradiction. Thus, $\tau$ is indeed critical. We verify  \cref{prop:path_truegap} (a)-(c) for $m_i$ and $\sigma$ below:
\begin{enumerate}
    \item[(a)]  there exists $m_j \in \tau \cap M_i \subset \sigma \cap M_i $ and $m_k \in \tau \cap M_{i+p-1} \subset \sigma \cap M_{i+p-1}$ such that $j<i<k\leq n-2p$ and $k-j \leq p$.  So, \cref{prop:path_truegap} (a) holds for $m_i$ and $\sigma$.
    \item[(b)] $\tau \cap M_{i+p-1} = \{m_k \} = \sigma \cap M_{i+p-1}$ since $k\leq i+p-1< n-p-1$. This means \cref{prop:path_truegap} (b) holds for $m_i$ and $\sigma$.
    \item[(c)] Since $i<n-2p$, we have $i+p < n-p+1$. We need $m_{i+p}\notin \sigma$ to verify \cref{prop:path_truegap} (c) for $m_i$ and $\sigma$. It suffices to show $m_{i+p}\notin \tau$ since $i+p < n-p+1$. If $i+p\leq n-2p$, then $m_{i+p}\notin \tau$  \cref{prop:path_truegap} (c) since $m_i$ is a true gap of $\tau$. If $i+p> n-2p$, then $m_{i+p}\notin \tau$  since $\tau$ is  a cell of $I_{p}(L_{n-(p+1)})$. So, $m_{i+p}\notin \tau$.
\end{enumerate}

Now, suppose $\tau$ is a critical cell of $I_p(L_{n-(p+1)})$. Let $\delta:= \tau \cup \{m_{n-p+1}\}$ and observe that $\delta$ is a cell of $I_p(L_n)$. Note that $\delta$ has no bridges. If it has a bridge $m_i$, then $m_i\neq m_{n-p+1}$ and $m_i \in \tau$. Since  $(n-p+1)-j\geq p+1$ for each $m_j\in \tau$ as $1\leq j \leq n-2p$, $m_i$ must be a bridge of $\tau$, contradiction. If $\delta$ has a true gap, say $m_i$, then $m_i$ does not dominate any bridges of $\delta$. So, we can apply \cref{prop:path_truegap}  to $m_i$ and $\delta$ to conclude that $m_i$ is a true gap of $\tau$ which leads to a contradiction.  Thus, $\delta$ is critical. We verify \cref{prop:path_truegap} (a)-(c) for $m_i$ and $\tau$ below:
\begin{enumerate}
    \item[(a)]  there exists $m_j \in \tau\cap M_i $ and $m_k \in \delta \cap M_{i+p-1} $ such that $j<i<k\leq i+p-1$ and $k-j \leq p$. Notice that $m_k \in \tau$ because if $k=n-p+1$, then $(n-p+1)-j \geq p+1$. So, $i<k\leq n-2p$ which means $m_i \in \tau$ and \cref{prop:path_truegap} (a) holds for $m_i$ and $\tau$.
    \item[(b)] $\delta \cap M_{i+p-1} = \{m_k \}=\tau  \cap M_{i+p-1} $ by part (b). So, \cref{prop:path_truegap} (b) holds for $m_i$ and $\tau$.
    \item[(c)] If $i+p \leq n-2p <n-p+1$, then $m_{i+p} \notin \delta$ implies that $m_{i+p} \notin \tau$. This means  \cref{prop:path_truegap} (c) holds for $m_i$ and $\tau$.
\end{enumerate}

\textbf{Case 3:} Suppose both \( m_{n-p+1} \) and \( m_{n-p} \) are in \( \sigma \). Our goal is to show the following:  $\sigma$ is a critical cell of $I_{p}(L_{n})$ if and only if $\sigma \backslash \{m_{n-p+1},m_{n-p}\}$ is a critical cell of $I_{p}(L_{n-(p+1)})$.

Its proof is almost identical to the proof of Case 2 and we only highlight the differences for the reader.  If $\sigma$ is a critical cell of  $I_{p}(L_{n})$, let $\tau:=\sigma \backslash \{m_{n-p+1}\}$. As in Case 2, $\tau$ is a cell of  $I_{p}(L_{n-(p+1)})$; otherwise,   $m_{n-p}$ is a bridge of of $\sigma$ by \cref{prop:bridge_gap}, which is not possible. The rest follows similarly as in this part of Case 2 and this completes the proof that $\tau$ is a critical cell of  $I_{p}(L_{n-(p+1)})$.

Let $\tau$ be a critical cell of $I_{p}(L_{n-(p+1)})$. Then  $\delta := \tau \cup \{m_{n-p}, m_{n-p+1}\}$  is a cell of $I_p(L_n)$.  We prove it is also a critical cell following the steps in the proof of Case 2. The only difference is in part (a) where we need to consider the possibility of $m_k=m_{n-p}$.  We recommend the reader to follow  along  part (a) of Case 2 to keep track of indexes that are referenced here. First notice that  $j<i\leq n-2p+1$  since $k\neq n-p+1$. If $m_k=m_{n-p}$, then $k-j \geq p$ as $j\leq n-2p$. Since $k-j\leq p$ by \cref{prop:path_truegap} (a), we must have $k-j=p$, indicating that $j=n-2p$ and $i=n-2p+1$. Then $m_{i+p}=m_{n-p+1} \in \delta$ which contradicts \cref{prop:path_truegap} (c) as we assumed $m_i$ is a true gap of $\delta$. Thus, $m_k\neq m_{n-p}$ which means $m_k\in \tau$. The rest of the proofs follows similarly to that of Case 2.
\end{proof}

Barile-Macchia resolutions are cellular, i.e., they are supported on CW complexes. We first remark that the dimension of the CW complex that supports the minimal resolution of a monomial ideal equals its projective dimension. In general, the minimal resolution of any monomial ideal of projective dimension $1$ is supported on a tree \cite[Theorem 1]{FH17}. In fact, in the cases where the path ideals of paths have projective dimension $1$, their minimal resolutions are supported on paths, which can be shown using the  techniques that will be employed in the next example. In what follows, we provide an example  where the path ideal of a path has projective dimension 2.

\begin{example}\label{ex:1}
Consider the path ideal \(I=I_p(L_{2p+1})\) under the total order 
\[ m_1\succ m_2\succ \cdots\succ  m_{p+2}. \]
For any subset \(\sigma = \{ m_{i_1}, \ldots, m_{i_k}\} \), each element of $\sigma$ (except \(m_{i_1}\) and \(m_{i_k}\)) is a bridge of \(\sigma\) due to \cref{prop:bridge_gap} (1). Consequently, \(m_{i_{k-1}}\) emerges as the smallest bridge of $\sigma$. Therefore,  the Barile-Macchia matching  of \(I\) with respect to \((\succ)\), denoted by \(A\), is achieved by removing the penultimate element at every iteration. Note that there is only one cell (of cardinality of at least 3) with no bridges:  \(\{m_1, m_{p+1}, m_{p+2}\}\). This results in the following list of all  critical cells of \(I\):
\[ \emptyset, \{m_1\}, \dots, \{m_{p+2}\}, \{m_1, m_2\}, \{m_2,m_3\}, \dots, \{m_{p+1}, m_{p+2}\}, \{m_{p+2}, m_1\}, \{m_1, m_{p+1}, m_{p+2}\}. \]
Given two distinct critical cells, \(\sigma\) and \(\sigma'\),  their least common multiples are different  by \cref{prop:content}. This uniqueness ensures that the Barile-Macchia resolution of \(I\) is minimal by \cref{thm:morseres}. 

By discrete Morse theory, there exists a cellular complex that supports this minimal free resolution, and it has $(p+2)$ many $0$-cells, $(p+2)$ many $1$-cells, and one $2$-cell. One obvious object that fits this description is the following $(p+2)$-gon:

\begin{center}
    \begin{tikzpicture}[scale = 2.5]
        \fill[fill = gray!30]
        (0, 0) -- (1, 0) -- (1.5, -0.7) -- (1, -1.4) -- (0, -1.4) -- (-0.5, -0.7) -- (0, 0);
        \draw[ultra thick, dotted]
        (-0.167, -1.167) -- (-0.333, -0.933);
        \draw
        (0, 0) -- node[above]{$\{m_1, m_2\}$} (1, 0) -- node[above right]{$\{m_2, m_3\}$} (1.5, -0.7) -- node[below right]{$\{m_3, m_4\}$} (1, -1.4) -- node[below]{$\{m_4, m_5\}$} (0, -1.4)
        (0, 0) -- node[above left]{$\{m_1, m_{p + 2}\}$} (-0.5, -0.7);
        \filldraw[fill = white]
        (0, 0) circle (1pt) node[above]{$\{m_1\}$}
        (1, 0) circle (1pt) node[above]{$\{m_2\}$}
        (1.5, -0.7) circle (1pt) node[right]{$\{m_3\}$}
        (1, -1.4) circle (1pt) node[below]{$\{m_4\}$}
        (0, -1.4) circle (1pt) node[below]{$\{m_5\}$}
        (-0.5, -0.7) circle (1pt) node[left]{$\{m_{p + 2}\}$};
        \draw
        (0.5, -0.6) node[below]{$\{m_1, m_{p + 1}, m_{p + 2}\}$};
    \end{tikzpicture}
\end{center}

Denote this  $(p+2)$-gon by $\Delta$. It follows from \cite[Lemma 2.2]{BPS98} that $\Delta$ supports a free resolution of $R/I$ if, for each monomial $m$, the subcomplex $\Delta[m] \coloneqq \{\sigma \in \Delta \colon \lcm(\sigma) \text{ divides } m\}$ of $\Delta$ is either empty or contractable. Notice that any  subcomplex $\Delta[m]$ is either empty, a line segment, or the $(p+2)$-gon itself for each squarefree monomial $m$. Since the last two objects are both contractable, $\Delta$ supports a free resolution of $R/I$.
\end{example}

\section{Minimal free resolutions of path ideals of cycles}\label{sec:cycle}

In this section, we turn our attention to path ideals of cycles, demonstrating that these ideals have minimal cellular resolutions. While in the preceding section we derived this outcome for path ideals of paths through identifying a minimal Barile-Macchia resolution of $R/I$ with respect to a specific total order on $\G(I)$, this approach falls short for path ideals of cycles. For instance, when we consider the edge ideal of a 9-cycle, it has no minimal Barile-Macchia resolutions as pointed out in \cite[Remark 4.23]{chau2022barile}.

In our investigation of path ideals of cycles, we transition our focus towards the generalized Barile-Macchia resolutions. These are Morse resolutions and can be considered as an extension of the  Barile-Macchia resolutions, introduced in \cite{chau2022barile}. The crux of these resolutions lies in utilizing a collection of total orders on $\G(I)$, instead of one. For the reader's convenience, we restate the construction of generalized Barile-Macchia resolutions from \cite{chau2022barile} along with a theorem stating that they induce cellular free resolutions. First, recall that $G_X$ denotes the directed graph obtained from the Taylor complex of $I$.
 
\begin{theorem}\label{thm:genBM}\cite[Theorem 5.19]{chau2022barile}
    For a monomial $u \in R$, let $G_u$ be the induced subgraph of $G_X$ on the vertices $\sigma \subseteq \G(I)$ where $\lcm (\sigma)=u$. Consider a total order $(\succ_u)$ on $\G(I)$ for each monomial $u\in R$. 
    
    Let $A$ be the  union of all $A_u$, where $A_u$ is the collection of directed edges obtained by applying the Barile-Macchia Algorithm to $G_u$ with respect to $(\succ_u)$ for each monomial $u\in R$. Then, $A$ is a homogeneous acyclic matching of $I$. The Morse resolution induced by $A$ is called the \textbf{generalized Barile-Macchia} resolution of $R/I$ with respect to $(\succ_u)_{u\in R}$.
\end{theorem}

Consider a cycle $C_n$ on $n$ vertices $\{x_1,\ldots, x_n\}$. Let \(R = \Bbbk[x_1, \ldots, x_n]\). The \(p\)-path ideal of \(C_n\), denoted as $I_p(C_n)$, is generated by monomials in $R$ corresponding to paths on $p$ vertices along $C_n$ where
\[I_p(C_n) = (x_1\cdots x_p, \ldots, x_n \cdots x_{p-1}).\]
Let \(m_i = x_i \cdots x_{i+p-1}\) for each  \(1\leq i\leq  n\) and consider the indices modulo $n$.   For the remainder of this section, we denote the  minimal generating set of $I_p(C_n)$ by $\G$ where
$$\G=(m_1, \dots, m_n).$$  
We restrict our attention to \(2 \leq p \leq n\) and  begin our discussion with an observation on the construction of the set \(A\), as outlined in  \cref{thm:genBM}, leveraging our previous findings on path ideals of paths.

\begin{observation}\label{obs:0}
Let \(A\) be a homogeneous acyclic matching constructed in accordance with \cref{thm:genBM} for the ideal \(I_p(C_n)\). To obtain a generalized Barile-Macchia resolution of \(R/I_p(C_n)\), it is necessary to determine its \(A\)-critical cells. 

\begin{enumerate}
    \item The following is immediate from the definition of $A$:
    \[
    \{A\text{-critical cells} \} = \bigcup_{u: \text{ monomial in } R}\{ A_u\text{-critical cells} \}.
    \] 
    \item Assuming \(u \neq x_1\cdots x_n\), remember that each vertex of \(G_X\) is a cell of $I_p(C_n)$, namely, a subset of \(\G\). For an induced subgraph $G_u$ of $G_X$,  its vertex set $V(G_u)$  is  empty or every vertex of \(G_u\) is a subset of \(\G(J)\) where \(J\) is the path ideal of some path. When \(V(G_u)\) is non-empty,  \cref{prop:content} assures the existence of a total order \( (\succ) \) on \(\G(J)\) that allows for precisely one \(A\)-critical cell.
\end{enumerate}

In light of the above observation, our attention is on the \(A_u\)-critical cells where \(u = x_1\cdots x_n\) to derive a minimal generalized Barile-Macchia resolution of \(R/I_p(C_n)\).
\end{observation}

For the remainder of this section, let $u = x_1 \cdots x_n$ and adopt the following total order $(\succ)$ for ~$u$: 
\[ m_1 \succ m_2 \succ \cdots \succ m_n. \]
Our primary objective is to demonstrate that all $A_u$-critical cells have the same cardinality. Consequently,  the resulting generalized Barile-Macchia resolution is minimal by \cref{thm:morseres} and \cref{obs:0}.

As a preliminary step, we examine the structure of  $\sigma \in V(G_u)$. Recall that $\lcm (\sigma) = u$ for each vertex $\sigma$ in $G_u$. For the remainder of this section, we express $\sigma$ based on the total order $(\succ)$. In other words, when $\sigma=\{m_{i_1},\dots, m_{i_t} \}$ we have $ m_{i_1} \succ \dots \succ m_{i_t}$; equivalently, $i_1<\cdots < i_t$.

\begin{proposition}\label{prop:verticesGm}
  Let $\sigma=\{m_{i_1},\dots, m_{i_t} \}$ be a vertex of $ G_u$. Then  the distance between consecutive elements of $\sigma$ is at most $p$. This means $i_{j+1} - i_j \leq p$ for all $j \in \{1,\ldots, t-1\}$ and $(i_1 + n)- i_t \leq p$. 
\end{proposition}

\begin{proof}
    Suppose that either $i_{j+1} - i_j > p$ for some $j \in \{1,\ldots, t-1\}$ or $i_1 + n - i_t > p$. In the first case, we have $m_{i_j}= x_{i_j} \cdots x_{i_j+p-1}$ and $m_{i_{j+1}}= x_{i_{j+1}} \cdots x_{i_{j+1}+p-1}$ where $i_{j+1}>i_j+p$. This  implies that $x_{i_j+p}$  does not divide $\lcm(\sigma)=u=x_1\cdots x_n$ which is not possible. The other case follows similarly by arguing  $x_{i_t+p}$ does not divide $u$, leading to contradiction.
\end{proof}

Unlike paths in the previous section,  every critical cell contains $m_n$ for cycles. 

\begin{lemma}\label{prop:Am-critical}
Monomial $m_n$ is contained in each  $A_u$-critical cell.
\end{lemma}

\begin{proof}
Let $\sigma$ be an $A_u$-critical cell. If $m_n \notin \sigma$, then  $m_n$ is a gap of $\sigma$ since $\lcm(\sigma)=\lcm (\sigma \cup m_n)$. Indeed, any $m\notin \sigma$ is a gap due to same reasoning. Notice that this means $m_n$ is a bridge of $\sigma \cup m_n$ and we have $\sbridge (\sigma \cup m_n) =m_n$  since $m_n$ is the smallest monomial among those in $\G$ with respect to $(\succ)$. It follows from \cref{rem:pot-type-2} that $\sigma \cup m_n$ is a potentially-type-2 cell. In fact, this cell is type-2 since there is no cell $\tau$ such that $\tau \backslash \sbridge(\tau)= \sigma$ with $m_n \succ  \sbridge(\tau)$. Then, the directed edge $(\sigma \cup m_n, \sigma)$ is contained in  $A_u$ which means  $\sigma$ cannot be a critical cell, a contradiction. Thus, $m_n \in \sigma$.
\end{proof}

\begin{notation}\label{rem:not_cycle}
    Similar to  \cref{not:Ms} for paths, define \( M_i \) to be the collection of all monomials in \( \G \) that are divisible by \( x_i \) for each $1\leq i\leq n$. Note that  \( |M_i| = p \) for each value of \( i \) and 
    $$
    M_i =
    \begin{cases}
       \{m_{i-p+1}, \ldots, m_i\} & \text{ if } i\geq p,\\
       \{ m_{(n+i) -p+1}, \ldots, m_n, m_1, \ldots, m_i \} &\text{ if } 1\leq i<p.
    \end{cases}$$
\end{notation}

In the subsequent discussion, we  analyze the vertices of \(G_u\) and describe their bridges, gaps, and true gaps. Analogous to the path case, we can classify the bridges in a similar manner. The proof is omitted since it directly follows from \cref{lem:tiny}, keeping in mind that indices are now considered modulo \(n\). First, we consider gaps as their classification is immediate.

\begin{proposition}\label{prop:cycle_gap}
  Let $\sigma \in V(G_u)$. A monomial $m_i$ is a gap of $\sigma$ if and only if $m_i \notin \sigma$.
\end{proposition}

\begin{proof}
   Since $\lcm (\sigma) = x_1\cdots x_n$,  adding $m_i$ to $\sigma$ does not change the least common multiple for any $m_i \notin \sigma$. So, any $m_i$ that is not contained in $\sigma$ is a gap.
\end{proof}

\begin{proposition}\label{cor:cycle_bridge}
    Let $\sigma \in V(G_u)$. Then the monomial $m_i$ is a bridge of $\sigma$ if and only if $m_i \in \sigma$ and there exist monomials $m_j$ and $m_k$ in $\sigma\setminus m_i$  for $1\leq j<k\leq n$ such that the distance between these two monomials along $m_i$ is at most $p$, i.e., $k-j \leq p$ when $j < i < k$ and $(j+n)-k \leq p$ when $j<k<i<j+n$. 
\end{proposition}

\begin{figure}[ht]
    \centering
    \includegraphics[width=0.5\linewidth]{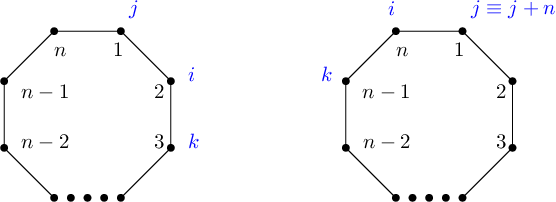}
    \caption{Only possible orderings of $i,j,k$ along $C_n$: $j<i<k$ (left) and  $j<k<i<j+n$ (right).}
    \label{fig:1}
\end{figure}
\begin{proof}
 If $m_i$ is a bridge of $\sigma$, one can apply the steps in the proof of \cref{lem:tiny} to $\sigma \setminus m_i$ (while keeping in mind that indices are now modulo $n$) and  obtain consecutive monomials  $m_j$ and $m_k$ in $\sigma\setminus m_i$  where either $j<i<k$ or $k<j<i<k+n$. Since  $\sigma\setminus m_i$ is still a vertex of $G_u$, distance between consecutive elements of $\sigma\setminus m_i$ is at most $p$ by \cref{prop:verticesGm}. For the other direction, if the conditions are met, $\sigma \setminus m_i$ is still a vertex of $G_u$ and $m_i$ divides $\lcm(m_j,m_k)$. Since $\lcm(m_j,m_k)$ divides $\lcm(\sigma \setminus m_i)$, monomial $m_i$ is a bridge of $\sigma$.
\end{proof}

\begin{remark}
    When we use  \cref{cor:cycle_bridge} in the rest of the paper, in order to make clear which monomials serve as $m_j,m_i,m_k$, we might refer to them as a triple $(m_j,m_i,m_k)$ where $m_i$ is the bridge in this triple.  It is possible to have $j<i<k$ or $k<j<i<k+n$ for this triple. Here, $m_j$ is closer to $m_n$ from the left and $m_k$ is closer to $m_n$ from the right. 
\end{remark}

We now turn our attention to characterizing the true gaps for vertices in $G_u$ that are $A_u$-critical.

\begin{proposition}\label{prop:cycle_truegap}
     Let $\sigma\in V(G_u)$ be an $A_u$-critical cell and let $m_i \in \G$. Assume $m_i$ does not dominate any bridge of $\sigma$. Then $m_i$ is a true gap of $\sigma$ if and only if the following statements hold:
        \begin{enumerate}
            \item[(a)] $m_i \notin \sigma$. 
            \item[(b)] If there exists a monomial $m_k \in \sigma \cap M_{i+p-1}$, then  $ \sigma \cap M_{i+p-1} =\{m_k\}$ and  $m_{i+p}\notin \sigma$.
            \item[(c)] If $1\leq i \leq p-1$ and $\sigma \cap \{m_1,\ldots, m_{i-1}\}=\emptyset$, then $\sigma \cap M_i = \{ m_n\}$ and  $m_{n+i-p}\notin \sigma$.
        \end{enumerate} 
\end{proposition}

\begin{proof}
  Suppose $m_i$ does not dominate any bridges of $\sigma$. Recall from \cref{prop:Am-critical} that  $m_n \in \sigma$ as $\sigma$ is an $A_u$-critical cell.  Lastly, as in the proof of \cref{prop:path_truegap}, it is useful to recall  \cite[Proposition 2.21]{chau2022barile}:  \( m_i \) is a true gap of $\sigma$ that does not dominate any bridges of $\sigma$ if and only if  \( m_i \)  is a gap of \( \sigma \) and \( \sbridge(\sigma \cup m_i) = m_i \).

    For the forward direction, assume $m_i$ is a true gap of $\sigma$. So, \( \sbridge(\sigma \cup m_i) = m_i \) from the previous paragraph.  First, it is immediate that $m_i$ is a gap of $\sigma$ which means $m_i \notin \sigma$ by \cref{prop:cycle_gap}. So, (a) holds.     For (b), if there are two monomials $m_k,m_t \in \sigma \cap M_{i+p-1}$, then we can reorder the vertices of $C_n$ so that $i<k,t\leq  i+p-1$. As in the proof of \cref{prop:path_truegap} (b), we can assume $k<t$ and  conclude that $m_t$ is a bridge of $\sigma \cup m_i$ where $m_i\succ m_t$, which contradicts to \( \sbridge(\sigma \cup m_i) = m_i \). For the second part of (b), notice that,  if $m_{i+p}\in \sigma$, there exists $m_j \in \sigma$ where $i<j<i+p$ by \cref{prop:verticesGm} (by reordering vertices, if needed). This means $m_j$ is a bridge of $\sigma \cup m_i$ with $m_i\succ m_j$, a contradiction since \( \sbridge(\sigma \cup m_i) = m_i \). Thus, $m_{i+p}\notin \sigma$.
    
    Now, to verify (c), assume that $1 \leq i \leq p-1$ and   $\sigma \cap \{m_1,\ldots, m_{i-1}\} =\emptyset$. Given that $m_n \in M_i$ for the stated range of $i$, we have $m_n \in \sigma \cap M_i$. If there exists another monomial $m_t \in \sigma \cap M_i$  with $t \notin \{1, \ldots, i\}$, the monomial $m_n$ would emerge as a bridge for $\sigma \cup m_i$ by \cref{cor:cycle_bridge} for the triple $(m_t, m_n, m_i)$ since $(n+i)-p+1< t$. This implies that $\sbridge(\sigma \cup m_i) = m_n$, a contradiction. Hence,  $\sigma \cap M_i =\{ m_n\}$. For the last segment of (c), if we assume $m_{n+i-p} \in \sigma$, it results with $m_n$ as a bridge of $\sigma\cup m_i$ by \cref{cor:cycle_bridge}  for the triple $(m_{n+i-p}, m_n,m_i)$. This would imply that $\sbridge(\sigma \cup m_i) = m_n$, which is, yet again, contradictory to $\sbridge(\sigma \cup m_i) = m_i$, thereby confirming that $m_{n+i-p} \notin \sigma$.

For the reverse direction, assume  (a), (b), and (c) hold. We argue by contradiction and assume that $m_i$ is not a true gap of $\sigma$. Then there exists a monomial  $m_t \in \sigma$ that is a bridge of $\sigma \cup m_i$ but is not a bridge of $\sigma$, with $m_i \succ m_t$.  Furthermore, by \cref{cor:cycle_bridge}, there exist monomials $m_{j_t}$ and $m_{k_t}$  in $ (\sigma  \setminus m_t)\cup m_i$   such that
 $k_t-j_t\leq p$ when $j_t <t<k_t$ or  $(j_t+n)-k_t\leq p$ otherwise. 

Note that  either $j_t = i$ or $k_t = i$ since  $m_t$ is not a bridge of $\sigma$. As in the proof \cref{cor:cycle_bridge}, we may assume that $m_{j_t}$ and $m_{k_t}$ are consecutive monomials in  $(\sigma \setminus m_t) \cup m_i  $ that are closest to $m_t$. If $j_t < t < k_t$, we must have $j_t = i$ since  $m_i \succ m_t$. The reasoning for this case mirrors the corresponding portion of  \cref{prop:path_truegap}, eventually leading to a contradiction.

For the remaining case (2),   we have:
\begin{equation}\label{eq:ineq1}
    j_t<k_t<t<j_t+n\leq k_t+p\leq n+p. \tag{$\star$}
\end{equation}   
Since   $i < t$, either $k_t=i$ or $j_t=i$. If $k_t = i$, then $m_{j_t} \in \sigma$.  It follows from (\ref{eq:ineq1}) that $m_t\in \sigma \cap M_{i+p-1}$. Then by (b) and the fact that $m_{j_t}=m_{j_t+n} \in \sigma$, we have  $j_t+n=i+p$ which means  $m_{j_t+n} = m_{i+p} \in \sigma$, a contradiction to (b). 

For the final case, we can assume $j_t=i$ and we have $m_{k_t}\in \sigma$ in this case. Next, we verify that conditions of (c) are satisfied. Notice that $\sigma \cap \{m_1, \ldots, m_{i-1}\}=\emptyset$ since $j_t$ and $k_t$ are chosen to be the closest monomials to $m_t$ in $(\sigma \setminus m_t) \cup m_i$. In addition, $i < p$ as a consequence of $i+n\leq k_t+p\leq n+p$ from (\ref{eq:ineq1}). Then  (c) holds, namely, $\sigma \cap M_i=\{m_n\}$ where 
$M_i=\{ m_{(n+i) -p+1}, \ldots, m_n, m_1, \ldots, m_i \}.$ Notice that $m_t \in \sigma \cap M_i$ since  $(i+n)-p \leq k_t < t\leq n$.
Thus, $m_t=m_n$ and  $m_{k_t} = m_{n+i-p} \in \sigma$, a contradiction to (c). Therefore, $m_i$ is indeed a true gap of $\sigma$.
\end{proof}

Having addressed the true gaps of \( A_u \)-critical cells in the preceding proposition, our focus now shifts to the distinct classes of critical cells introduced in \cref{def:criticals}: the absolutely critical and the fortunately critical cells. As noted in  \cref{cor:abs_critical}, the absolutely critical cells are uniquely characterized by the lack of both bridges and true gaps. On the other hand, the fortunately critical cells stand out. Their first element serves as their smallest bridge while still having no true gaps.

\begin{lemma}\label{prop:criticalshape}
    Let $\sigma\in V(G_u)$ be an \(A_u\)-critical cell. If $\sigma$ is fortunately critical, then:
    \begin{enumerate}
        \item The smallest bridge of \( \sigma \) satisfies \( \sbridge(\sigma) \succ m_p \).
        \item If $\sigma=\{m_{i_1},\ldots, m_{i_t}\}$, then $\sbridge(\sigma)=m_{i_1}.$
    \end{enumerate}
    Furthermore,  no monomial from the set \( \{ m_p, m_{p+1}, \ldots , m_n \} \) serves as a bridge or a true gap of \( \sigma \) (irrespective of whether $\sigma$ is fortunately or absolutely critical).
\end{lemma}

\begin{proof}
    Let $\sigma$ be a vertex in $G_u$ that is fortunately critical. Then $\sigma$ is potentially-type-2 and, by definition, there exists another vertex $\sigma'$ such that $\sigma\backslash \sbridge(\sigma)=\sigma'\backslash \sbridge(\sigma')$ and $\sbridge(\sigma)\succ \sbridge(\sigma')$.

  Since $\sigma$ is potentially-type-2,  every true gap of \( \sigma \) dominates \( \sbridge(\sigma) \) by \cref{rem:pot-type-2}. So,  \( \sbridge(\sigma') \) cannot be a true gap of \( \sigma \) as $\sbridge(\sigma)\succ \sbridge(\sigma')$. Then  there exists a monomial \( m_i \in \sigma \) for which \( \sbridge(\sigma \cup \sbridge(\sigma')) = m_i \), with \( \sbridge(\sigma') \succ m_i \) while \( m_i \) is not a bridge of \( \sigma \). Note as well that $m_i \in \sigma'$ and it is not a bridge of $\sigma'$. Before proceeding further, we first state our goal: to show that \( m_i = m_n \).

  By \cref{cor:cycle_bridge}, there exist monomials $m_j$ and $m_k$ in  $\sigma \cup \sbridge(\sigma')$ with $1\leq j<k\leq n$  and $i\notin \{j,k\}$ such that 
       \begin{enumerate}
           \item[(a)] $k-j \leq p$ when  $j<i<k$, or
           \item[(b)] $(j+n)-k \leq p$ when $j<k<i<j+n$.
       \end{enumerate}
Note that $\sigma$ and $\sigma'$ overlap at every element except $\sbridge(\sigma)$ and $\sbridge(\sigma')$. In other words, 
$$\sigma \setminus \{\sbridge(\sigma),\sbridge(\sigma')\} =\sigma'  \setminus \{\sbridge(\sigma),\sbridge(\sigma')\}.$$
 Inevitably, we have $\{\sbridge(\sigma),\sbridge(\sigma')\}= \{m_j, m_k\}$. Otherwise,  $m_i$ is a bridge of $\sigma$ and $\sigma'$, which is not possible. 

Notice that \( j<i<k \) is not possible since $\sbridge(\sigma)\succ \sbridge(\sigma') \succ m_i$. So, we must have $j<k<i<j+n$ as in (b) above.  Consequently, $m_j=\sbridge(\sigma)$ and $m_k=\sbridge(\sigma')$. Since  $m_n\in \sigma\cup \sbridge(\sigma')$, we conclude that $m_n$ is a bridge of $\sigma \cup \sbridge(\sigma')$ from \Cref{cor:cycle_bridge} by using the triple $(m_k,m_n,m_{j+n})$ and (b). This means $m_i=m_n$.
   
     \begin{enumerate}
      \item The proof of \( \sbridge(\sigma) \succ m_p \) follows from the following simple observation: the inequality \( j+n \leq k+p \leq n+p \) from (b) implies that \( j < p \).  Moreover, it is not possible to have \( j = p \), as this would imply \( k = n \), leading to \( \sbridge(\sigma') = m_n \), which is a contradiction. Thus, we conclude that \( j < p \), which is equivalent to \( m_j = \sbridge(\sigma) \succ m_p \).

    \item  Let  \( \sigma = \{ m_{i_1}, \ldots, m_{i_t} \} \) where $i_1<\cdots<i_t$. Our goal is to demonstrate  \( \sbridge(\sigma) = m_{i_1} \), i.e. $i_1=j$.  On the contrary, suppose there is a monomial \( m_s \) in \( \sigma \) such that \( m_s \succ \sbridge(\sigma) = m_{j} \). Then, monomial \( m_n \) is a bridge of \( \sigma' \) from \cref{cor:cycle_bridge}  by using the triple   \( (m_{k}, m_n, m_{s+n}) \) where each monomial is in \( \sigma' \) and  \( (s + n) - k < p \) by (b) above. This posits a contradiction since \( \sbridge(\sigma') =m_k\succ m_n \). Thus, \( \sbridge(\sigma) = m_{i_1} \).
  \end{enumerate}

  For the final part of the statement, consider a vertex $\sigma$ in $G_u$. If $\sigma$ is absolutely critical, it lacks both bridges and true gaps, by \cref{cor:abs_critical}, thereby satisfying the given statement. When \( \sigma \) is fortunately critical, the statement remains valid due to \( \sbridge(\sigma) \succ m_p \), and the fact that each true gap of \( \sigma \) dominates \( \sbridge(\sigma) \) by \cref{rem:pot-type-2}.
\end{proof}

Our primary objective is to comprehensively identify every element within an \(A_u\)-critical cell. We begin our identification with a series of observations and initiate the process by pinpointing specific values of \(j\) for which \(m_j \in \sigma\).

\begin{lemma}\label{lem:inside_sigma}
    Let $\sigma=\{m_{i_1},\ldots, m_{i_t}\}$ be an $A_u$-critical cell where $i_1<\cdots <i_t$. Then
\begin{enumerate}
    \item[(a)]  $i_t=n$, and $m_n$ is not a bridge for $\sigma$.
    \item[(b)] $i_{t-1}= n-k$ for some $1\leq k\leq p$.
    \item[(c)] $i_{t-2} < n-p$.
    \item[(d)] $i_1= p-k+1$ where $k$ is given as in (b).
\end{enumerate}
\end{lemma}

We need to make assumptions on the possible values of \( t \). For instance, to consider (b), we require \( t \geq 2 \); for (c), we need \( t \geq 3 \), and so on. However, since these assumptions are clear from the context, we omit them.

\begin{proof}
   Assume that $\sigma = \{m_{i_1}, \ldots, m_{i_t}\}$ is an $A_u$-critical cell.

\begin{enumerate}
    \item[(a)]  We begin by noting that $m_{i_t} = m_n$ as per \cref{prop:Am-critical}. If $\sigma$ is absolutely critical, then it has no bridges by \cref{cor:abs_critical}. On the other hand, for a fortunately critical $\sigma$, \cref{prop:criticalshape} forces that $\sbridge(\sigma) = m_{i_1}$. Given that $m_{i_1} \succ m_n$, it is evident that $m_n$ is not a bridge for $\sigma$. 
    \item[(b)] Recall from \cref{prop:verticesGm} that  $i_t-i_{t-1} \leq p$. Then, \( i_{t-1} = n-k \) for some \( 1 \leq k \leq p \). 
    \item[(c)] For the sake of contradiction, suppose \( m_{n-j} \in  \sigma \) for \( 1 \leq k< j \leq p \). Then, \( m_{n-k} \) is a bridge for \( \sigma \) from \cref{cor:cycle_bridge} by using the triple $(m_{n-j},m_{n-k}, m_n)$. However, this contradicts the nature of \( \sigma \), as it is either absolutely critical (hence having no bridges) or fortunately critical (where \( \sbridge(\sigma) = m_{i_1} \)). So, $i_{t-2} < n-p$.
    \item[(d)]    Our initial step is to derive  \( p-k+1 \leq i_1 \leq p \). The upper bound \(i_1 \leq p\) is a direct consequence of \cref{prop:verticesGm} since $(i_1+n)-i_t\leq p$. To establish the lower bound,  if there exists an \( m_i \in \sigma \) for \( i \leq p-k \), then \( m_n \) is a bridge of \( \sigma \) from  \cref{cor:cycle_bridge} by using the triple $(m_{n-k}, m_n, m_{n+p-k})$.  This assertion, however, yields a contradiction by (a). Thus, \( p-k +1 \leq  i_1 \leq p \).

    If $k=1$, then it is immediate that \(i_1 = p\), satisfying the statement of (d). For \(2 \leq k \leq p\), it remains to show \(i_1 \leq p-k+1\). A significant insight here is that \(m_{n-k+1}\) cannot be a true gap of \(\sigma\). This is immediate when $\sigma$ is absolutely critical as it has no true gaps. Suppose $\sigma$ is fortunately critical. Then  $\sbridge(\sigma)=m_{i_1}$ by \cref{prop:criticalshape} (2) and  every true gap of $\sigma$ dominates $\sbridge(\sigma)$ by \cref{rem:pot-type-2} as $\sigma$ is potentially-type-2. Since $\sbridge(\sigma)=m_{i_1} \succ m_{n-k} \succ m_{n-k+1}$, then  $m_{n-k+1}$ cannot be a true gap of $\sigma$. This means \(\sigma \cup m_{n-k+1}\) has a  bridge \(m_i \in \sigma\) such that \(m_{n-k+1} \succ m_b\). Since the only monomial in $\sigma$ that is dominated by $m_{n-k+1}$ is $m_n$, we have $m_i=m_n$. Using \cref{cor:cycle_bridge} for the triple $(m_{n-k+1},m_n,m_{n+i_1})$, we conclude that \(i_1 \leq p-k+1\). \qedhere
\end{enumerate}
\end{proof}

To identify  other elements of \( \sigma \), we examine them in relation to the possible values of \( k \). Here, \( n-k \) is the penultimate element of \( \sigma \) when \( 1 \leq k \leq p \). The next two lemmas address the \( k=1 \) case and the \( 2 \leq k \leq p \) case separately due to nuanced variations in their proofs. Together, these lemmas give a complete overview of all \( A_u \)-critical cells in \cref{prop:criticalofcycles}.

\begin{lemma}\label{lem:k=1}
  Let $\sigma=\{m_{i_1},\ldots, m_{i_t}\}$ be an $A_u$-critical cell with $i_{t-1}=n-1$. Then, the distance between consecutive elements of \(\sigma\) alternates between \(1\) and \(p\). Specifically, 
\[
\sigma = \{m_{p}, \ldots, m_{n-(p+1)}, m_{n-(p+1)}, m_{n-1}, m_n\}.
\] 
Moreover, we have \( p \equiv n\) or \(n-1 \pmod{p+1}\).
\end{lemma}
\begin{proof}
First note that \(i_1=p\) by \cref{lem:inside_sigma}.  Then,  \(\sigma\) is an absolutely critical cell by \cref{prop:criticalshape}.  

The main idea of the proof revolves around retracing our steps from \(i_{t-1}\), pinpointing the preceding indices of elements in \( \sigma \) until we arrive at \(i_1=p\). A recurring and instrumental point from \cref{prop:verticesGm} is:  
\( i_{j}-i_{j-1} \leq p  \) for \( m_{i_{j-1}}, m_{i_j} \in \sigma \). We reference this result as ($\star$) throughout the proof.

   Recall from \cref{lem:inside_sigma}  (c) that  \(i_{t-2}\leq n-1-p\), indicating $i_{t-1}-i_{t-2}\geq p$ as $i_{t-1}=n-1$. This  bound coupled with ($\star$) yields to \(i_{t-1}-i_{t-2}= p\). So, the first step of our analysis is complete. If  \(i_{t-2} = i_{1}\), we are done; otherwise, we move on to \(i_{t-3}\).

    Next, we show that \(i_{t-2}-i_{t-3} =1\) which is equivalent to  obtaining \(m_{n-p-2} \in \sigma\).  For the sake of contradiction, suppose   \(m_{n-p-2} \notin \sigma\). Since $\sigma$ is absolutely critical, $m_{n-p-2}$ cannot be a true gap of $\sigma$. This means $\sigma \cup m_{n-p-2}$  has a  bridge \(m_b \in \sigma\)  such that  \(m_{n-p-2} \succ m_b\) and $m_b $ is not a bridge of $\sigma$. Then, by  \cref{cor:cycle_bridge}, there exists $m_s\in \sigma$ such that  the distance between $m_s$ and  $m_{n-p-2}$ along $m_b$ is at most $p$. Given that \(m_{n-p-2} \succ m_b\),  we have one of the following scenarios by  \cref{cor:cycle_bridge}: 
    \begin{enumerate}
        \item $s- (n-p-2)\leq p$ when $1\leq n-p-2<b<s\leq n$, or
        \item $s+n -(n-p-2)\leq p$ when  $1\leq s<n-p-2<b \leq n.$
    \end{enumerate}
    The latter case is not possible because it implies that $s\leq -2$. Thus, (1) holds.  In this case,  we have $m_b=m_{n-1}$ and $m_s=m_n$ since $m_b, m_s  \in \sigma$. Then, (1) results with $p+2\leq p$, a contradiction. Thus, we conclude that \(m_{n-p-2} \in \sigma\). So, \(i_{t-3} = n-p-2\). If  \(i_{t-3}=i_{1}\), our task is complete; otherwise, our focus shifts to \(i_{t-4}\).

    Towards showing \(i_{t-3}-i_{t-4} = p\), recall that   \(i_{t-3}-i_{t-4} \leq p\) by ($\star$), or equivalently, \(i_{t-4} \geq n-2(p+1)\). If \( m_{n-2(p+1)} \notin \sigma\), then $m_{i_{t-3}}$ is a bridge of $\sigma$ by   \cref{cor:cycle_bridge}  since $i_{t-2}-i_{t-4} \leq p$. This leads to a contradiction because $\sigma$ is absolutely critical.  Hence, $i_{t-4}= n-2(p+1)$. If  \(i_{t-4}=i_{1}\), our investigation is complete; otherwise, we continue in the same fashion for $i_{t-5}$.

    Subsequent distances between consecutive elements of $\sigma$ can be concluded by employing similar arguments in an alternating way until we reach \(i_1=p\). This results in the congruence \(p \equiv n \text{ or } n-1 \pmod{p+1}\), concluding the proof.
\end{proof}

\begin{lemma}\label{lem:bigk}
  Let $\sigma=\{m_{i_1},\ldots, m_{i_t}\}$ be an $A_u$-critical cell with $i_{t-1}=n-k$ for some $2\leq k\leq p$. Then, the distance between consecutive elements of \(\sigma\) alternates between \(k\) and \((p+1)-k\). Specifically, 
\[
\sigma = \{m_{p-k+1}, \ldots, m_{n-k-(p+1)}, m_{n-(p+1)}, m_{n-k}, m_n\}
\] 
Moreover, we have \( p-k+1 \equiv n\) or \(n-k \pmod{p+1}\).
\end{lemma}
\begin{proof}
   The idea behind this proof is similar to that of \cref{lem:k=1}.  Recall  that \(i_{t-2} \leq n-p-1\) and \( i_1=p-k+1\) by \cref{lem:inside_sigma} (c) and (d). So, our first goal is to verify  \(i_{t-2} = n-p-1\). 

    Assume, for the sake of contradiction, that \(m_{n-p-1} \notin \sigma\). We then observe that \(m_{n-p-1}\) cannot be a true gap of \(\sigma\). This is evident if \(\sigma\) is absolutely critical because it lacks true gaps. If \(\sigma\) is fortunately critical, then  \(\sbridge(\sigma) = m_{i_1}\) as stated in \cref{prop:criticalshape} and any true gap of $\sigma$ dominates \(\sbridge(\sigma) = m_{i_1}\) by \cref{rem:pot-type-2}. If \(m_{n-p-1}\) is a true gap of $\sigma$, then we must have $\sigma=\{m_{n-k}, m_n\}$ and it cannot have any bridges, a contradiction. So,  \(m_{n-p-1}\) cannot be a true gap of \(\sigma\). This means \(\sigma \cup m_{n-p-1}\) has a bridge \(m_b \in \sigma \) such that \(m_{n-p-1} \succ m_b\) and $m_b$ is not a bridge of $\sigma$. Using \cref{cor:cycle_bridge}, there exists a monomial \(m_s \in \sigma\) such that the  distance between $m_s$ and  \(m_{n-p-1}\) along $m_b$ is at most \(p\). In other words, given that \(m_{n-p-1} \succ m_b\), we have
    \begin{enumerate}
        \item $s- (n-p-1)\leq p$ when $1\leq n-1-p < b<s \leq n$ or 
        \item $s+n- (n-p-1)\leq p$ when $1\leq s < n-1-p<b\leq n$.
    \end{enumerate}

The latter is not possible since it implies \(s < -1\). Thus, (1) holds. Since $ m_b, m_s \in \sigma$, we have  \(m_b = m_{n-k}\) and \(m_s = m_n\). Using these in (1) results  with $p+1<p$, a contradiction. Hence, \(i_{t-2} = n-1-p\). If \(i_{t-2}=i_1\), our investigation concludes.   Suppose \(i_{t-2} \neq i_1\).

      The next step is to show \(i_{t-3} = n-k-(p+1)\). We first claim that $m_{i_{t-2}}$ cannot be a bridge of $\sigma$. If it is a bridge of $\sigma$, then  $\sbridge(\sigma)=m_{i_1}$ by \cref{prop:criticalshape} because $\sigma$ must be fortunately critical. This means $i_1=i_{t-2}$, a contradiction to our assumption. Since $m_{i_{t-2}}$ cannot be a bridge of $\sigma$, we have $i_{t-3}\leq n-k-(p+1)$ by \cref{cor:cycle_bridge}.  So, it suffices to show $m_{n-k-(p+1)} \in \sigma$.
      
      For the sake of contradiction, suppose $m_{n-k-(p+1)} \notin \sigma$. If $m_{n-k-(p+1)}$ is a true gap of $\sigma$, then  $\sigma$ is potentially-type-2 and $m_{n-k-(p+1)}\succ \sbridge(\sigma)=m_{i_1}$ by \cref{rem:pot-type-2} and  \cref{prop:criticalshape}. This can happen only when $i_1=i_{t-2}$, a contradiction. Thus, $m_{n-k-(p+1)}$ is not a true gap of $\sigma$. This means  $\sigma \cup m_{n-k-(p+1)}$ has a  bridge $m_b \in \sigma$ such that $m_{n-k-(p+1)}\succ m_b$ and $m_b$ is not a bridge of $\sigma$. Then, by \cref{cor:cycle_bridge}, there exists $m_s\in \sigma$ such that the distance between $m_{n-k-(p+1)}$ and $m_s$  along $m_b$ is at most $p$. 
     In other words, given that \(m_{n-k-(p+1)} \succ m_b\),  we have
    \begin{enumerate}
        \item $s- (n-k-(p+1))\leq p$ when $1\leq n-k-(p+1)< b<s\leq n$ or 
        \item $s+n- (n-k-(p+1))\leq p$ when $1\leq s< n-k-(p+1)< b\leq n$
    \end{enumerate}
    One can verify that (2) cannot happen. Then, by (1),
    the monomial $m_s$ is either $m_{n-k}$ or $m_n$ since $m_s \in \sigma$. Using either value of $s$ in (1) results with  $p+1 \leq p$, a contradiction. Therefore, $i_{t-3}= n-k-(p+1)$. If $i_{t-3}=i_1$, the process terminates here. If $i_{t-3}\neq i_1$, one can repeat the above arguments for the next steps until reaching $i_1=p-k+1$.   This results in the congruence \(p -k+1\equiv n \text{ or } n-k \pmod{p+1}\), concluding the proof.
\end{proof}

Below, we describe all $A_u$-critical cells. This proposition serves as the centerpiece of this chapter's main result.

\begin{proposition}\label{prop:criticalofcycles}
 Let $n=(p+1)q+r$ where $0\leq r \leq p$. Then, we have:    \begin{enumerate}
        \item For $r=0$, the only $A_u$-critical cells are the cells $\sigma_i$, where 
        \[
        \sigma_i := \{m_j \mid j \equiv i \pmod{p+1}\}
        \]
        for each $1 \leq i \leq p$. Moreover, each $\sigma_i$ is absolutely critical.
        \item For $r \neq 0$, the only $A_u$-critical cell is 
        \[
        \tau_r := \{m_j \mid j \geq r, j \equiv r,2r \pmod{p+1}\}.
        \]
        Moreover, the cell $\tau_r$ is absolutely critical if and only if $2r > p$.
    \end{enumerate}
\end{proposition}

\begin{proof}
    Let $\sigma = \{m_{i_1}, \ldots, m_{i_t}\}$ be an $A_u$-critical cell with $t \geq 2$. To identify all $A_u$-critical cells, we first address the immediate case where $n = p+1$. It is evident that $t \geq 3$ is not possible, as this would lead to $\sbridge(\sigma) = m_n$, a contradiction by \cref{lem:inside_sigma} (a). Consequently, $\sigma$ must be of the form $\sigma_i = \{m_i, m_n\}$ for all $1 \leq i \leq p$. A straightforward check confirms that every such $\sigma_i$ is absolutely critical. Having settled this case, we proceed under the assumption $n > p+1$.

    Recall from \cref{lem:k=1} and \cref{lem:bigk} that the distance between consecutive elements of \(\sigma\) alternates between \(k\) and \((p+1)-k\). Specifically,  for $1\leq k\leq p$, we have
\[
\sigma = \{m_{p-k+1}, \ldots, m_{n-k-(p+1)}, m_{n-(p+1)}, m_{n-k}, m_n\}
\]
where \( p-k+1 \equiv n\) or \(n-k \pmod{p+1}\). It is important to clarify that the distance between the first and last elements of \(\sigma\) is not considered.

Observe that the conditions \( p-k+1 \equiv n \pmod{p+1}\) and \(n-k \equiv 0 \pmod{p+1}\) are equivalently expressed as \( n  \equiv p+1-k \pmod{p+1}\) and \( n  \equiv 0 \pmod{p+1}\), respectively. Consequently, the cell \( \sigma \) can be expressed as follows:
\[
   \sigma = 
   \begin{cases} 
       \{m_{(p+1-k)+i(p+1)}, m_{(i+1)(p+1)} ~:~ 0\leq i\leq q-1\} & \text{if } r=0, \\
       \{m_{r+i(p+1)}, m_{2r+i(p+1)} ~:~  0\leq i\leq q-1 \} \cup \{m_n\} & \text{if } r\neq 0,
   \end{cases}
\]
where the former corresponds to \( \sigma_{p+1-k} \) for \( 1\leq k\leq p \) and the latter corresponds to \( \tau_r \) in the statement of the proposition. One can verify that each \( \sigma_{p+1-k} \) is absolutely critical by noticing that  they have no bridges and no true gaps for \( 1\leq k\leq p \). Furthermore, one can verify that \( \tau_r \) is  absolutely critical if and only if \( 2r > p \).
\end{proof}

  The conclusions drawn below directly arise from the descriptions of \( A_u \)-critical cells given in \cref{prop:criticalofcycles}.

\begin{corollary}
    Let $\sigma\in V(G_u)$ be a critical cells of $\G$. Then
    \[
    |\sigma|=\begin{cases}
        2q & \text{ if } n\equiv 0 \pmod{(p+1)},\\
        2q+1 & \text{ otherwise}.\\
    \end{cases}
    \]
    In particular, all $A_u$-critical cells have the same cardinality.
\end{corollary}

We now present the main theorem of this section, which follows immediately from the preceding corollary, and the description of the differentials of the corresponding generalized Barile-Macchia resolution in \cref{thm:morseres}.

\begin{theorem}\label{thm:main_cycle}
    A generalized Barile-Macchia resolution of $I_p(C_n)$ is minimal. 
\end{theorem}

We conclude this section with a brief discussion. In \cite{chau2022barile}, the first two authors introduced and examined Barile-Macchia resolutions, which are cellular and independent of \(\characteristic(\Bbbk)\). Though effective for many classes of ideals, this method does not always produce a minimal free resolution. Specifically, the edge ideal of a 9-cycle, \(I_2(C_9)\), cannot be minimally resolved by a Barile-Macchia resolution, as highlighted in \cite{chau2022barile}.

\begin{corollary}
   Edge ideals of cycles have cellular minimal free resolutions.
\end{corollary}

Echoing the approach of the preceding section, we conclude with an example presenting a CW complex that supports minimal free resolutions of the path ideals of cycles, where the projective dimension is equal to $2$.

\begin{example}
    Consider the cycle $C_n$ where $n=(p+1)+r$ for $1\leq r\leq p$. We can list all  critical cells of \(I_p(C_n)\) using \cref{prop:criticalofcycles} as follows:
\[ \emptyset, \{m_1\}, \dots, \{m_{p+2}\}, \{m_1, m_2\}, \{m_2,m_3\}, \dots, \{m_{n-1}, m_{n}\}, \{m_{n}, m_1\}, \{m_r, m_{2r}, m_{n}\}.\]
 The following $n$-gon, denoted by $\Delta$, is in correspondence with the critical cells of \(I_p(C_n)\):

\begin{center}
    \begin{tikzpicture}[scale = 2.5]
        \fill[fill = gray!30]
        (0, 0) -- (1, 0) -- (1.5, -0.7) -- (1, -1.4) -- (0, -1.4) -- (-0.5, -0.7) -- (0, 0);
        \draw[ultra thick, dotted]
        (-0.167, -1.167) -- (-0.333, -0.933);
        \draw
        (0, 0) -- node[above]{$\{m_1, m_2\}$} (1, 0) -- node[above right]{$\{m_2, m_3\}$} (1.5, -0.7) -- node[below right]{$\{m_3, m_4\}$} (1, -1.4) -- node[below]{$\{m_4, m_5\}$} (0, -1.4)
        (0, 0) -- node[above left]{$\{m_1, m_{n}\}$} (-0.5, -0.7);
        \filldraw[fill = white]
        (0, 0) circle (1pt) node[above]{$\{m_1\}$}
        (1, 0) circle (1pt) node[above]{$\{m_2\}$}
        (1.5, -0.7) circle (1pt) node[right]{$\{m_3\}$}
        (1, -1.4) circle (1pt) node[below]{$\{m_4\}$}
        (0, -1.4) circle (1pt) node[below]{$\{m_5\}$}
        (-0.5, -0.7) circle (1pt) node[left]{$\{m_{n}\}$};
        \draw
        (0.5, -0.6) node[below]{$\{m_r, m_{2r}, m_{n}\}$};
    \end{tikzpicture}
\end{center}

As in \cref{ex:1},  all restricted subcomplexes of $\Delta$ are either empty, a line segment, or the $n$-gon itself, which are all acyclic. Thus, $\Delta$ support the minimal free resolution of $R/I_p(C_n)$ by \cite[Lemma 2.2]{BPS98}.
\end{example}

\textbf{Acknowledgements.} We thank Sara Faridi for helpful discussions about the manuscript. The first author was supported by NSF grants DMS 1801285, 2101671, and 2001368. The second author was supported by NSF grant DMS-2418805.

\bibliographystyle{abbrv}
\bibliography{refs}
\end{document}